\documentclass[leqno]{amsart}
\usepackage{amsfonts,amssymb,amsmath,amsgen,amsthm}
\usepackage{hyperref}
\usepackage{color}
\usepackage[normalem]{ulem}
\usepackage{empheq}
\usepackage{graphicx,subcaption}

\newcommand{\msc}[2][2000]{%
  \let\@oldtitle\@title%
  \gdef\@title{\@oldtitle\footnotetext{#1 \emph{Mathematics subject
        classification.} #2}}%
}

\theoremstyle{plain}
\newtheorem{theorem}{Theorem} [section]

\newtheorem{lemma}[theorem]{Lemma}
\newtheorem{corollary}[theorem]{Corollary}
\newtheorem{proposition}[theorem]{Proposition}
\newtheorem{hyp}[theorem]{Assumption}
\theoremstyle{remark}
\newtheorem{remark}[theorem]{Remark}

\newtheorem{example}[theorem]{Example}

\def\C{{\mathbb C}}
\def\R{{\mathbb R}}
\def\N{{\mathbb N}}
\def\Sch{{\mathcal S}}
\def\O{\mathcal O}

\def\op{\mathrm{op}}
\def\Re{\mathrm{Re}}
\def\Im{\mathrm{Im}}
\def\d{{\partial}}
\def\eps{\varepsilon}

\def\aaa{{1}}
\def\gg{{\tt g}}
\def\eigen{{\Lambda}}

\def\Tend#1#2{\mathop{\longrightarrow}\limits_{#1\rightarrow#2}}

\def\({\left(}
\def\){\right)}
\def\<{\left\langle}
\def\>{\right\rangle}
\def\le{\leqslant}
\def\ge{\geqslant}

\numberwithin{equation}{section}

\begin{document}

\title[Separation of scales]{Separation of scales: Dynamical approximations for composite quantum systems}

\author[I. Burghardt]{Irene Burghardt}
\address[I. Burghardt]{Institute of Physical and Theoretical Chemistry, Goethe University Frankfurt, Max-von-Laue-Str. 7, 60438 Frankfurt am Main, Germany }
\email{burghardt@chemie.uni-frankfurt.de}

\author[R. Carles]{R\'emi Carles}
\address[R. Carles]{Univ Rennes, CNRS\\ IRMAR - UMR 6625\\ F-35000
  Rennes, France}
\email{Remi.Carles@math.cnrs.fr}

\author[C. Fermanian]{Clotilde Fermanian Kammerer}
\address[C. Fermanian Kammerer]{
Univ Paris Est Creteil, CNRS, LAMA, F-94010 Creteil, France\\ 
Univ Gustave Eiffel, LAMA, F-77447 Marne-la-Vallée, France}
\email{clotilde.fermanian@u-pec.fr}

\author[B. Lasorne]{Benjamin Lasorne}
\address[B. Lasorne]{ICGM, Univ Montpellier, CNRS, ENSCM, Montpellier, France}
\email{benjamin.lasorne@umontpellier.fr}

\author[C. Lasser]{Caroline Lasser}
\address[C. Lasser]{Technische Universit\"at M\"unchen, Zentrum Mathematik, Deutschland}
\email{classer@ma.tum.de}


\begin{abstract}
We consider composite quantum-dynamical systems that can be partitioned into weakly interacting subsystems, similar to system-bath type situations. Using a factorized wave function ansatz, we mathematically characterize dynamical scale separation.
Specifically, we investigate a coupling régime that is partially flat, i.e., slowly varying with respect to one set of variables, for example, those of the bath. Further, we study the situation where one of the sets of variables is semiclassically scaled and derive a quantum-classical formulation. In both situations, we propose two schemes of dimension reduction: one based on Taylor expansion (collocation) and the other one based on partial averaging (mean-field). We analyze the error  
for the wave function and for the action of observables, obtaining comparable estimates for both approaches. The present study is the first step towards a general analysis of scale separation in the context of tensorized wavefunction representations.
\end{abstract}

\maketitle

\section{Introduction}
We consider \textcolor{black}{composite} quantum-dynamical systems that can be partitioned into weakly interacting subsystems, 
with the goal of developing effective dynamical descriptions that simplify the original, fully quantum-mechanical formulation.
Typical examples are small reactive molecular fragments embedded in a large molecular bath, namely, a protein, or a solvent, all being governed dynamically by quite distinct energy and time scales.
To this end, various r\'egimes  of intersystem couplings are considered, and a quantum-classical approximation is explored. A key aspect is dimension reduction at the wave function level, without referring to the conventional ``reduced dynamics" approaches that are employed in system-bath theories.

\subsection{The mathematical setting}

The quantum system is described by a time-dependent Schr\"odinger equation 
\begin{equation}\label{eq:sep-scales}
i\d_t \psi = H\psi\quad;\quad
   \psi_{\mid t=0}= \psi_0,
\end{equation}
 governed by a Hamiltonian of the form 
\[
H=H_x+H_y +W(x,y),
\]
where the coupling potential $W(x,y)$ is a smooth function, that satisfies growth estimates 
guaranteeing existence and uniqueness of the solution to the Schr\"odinger equation \eqref{eq:sep-scales} for a rather general set of  initial data, as we shall see later in Section~\ref{sec:hypothesis}.
The overall set of space variables is denoted as $(x,y) \in \R^n\times\R^d$ such that the 
total dimension of the configuration space is $n+d$. The wave function depends 
on time $t\ge0$ and both space variables, that is, $\psi = \psi(t,x,y)$. 
We suppose that initially scales are separable, that is, we work with initial data of product form 
\begin{equation}
  \label{eq:CI}
  \psi(0,x,y)=\psi_0(x,y) =    \varphi_0^x(x)\varphi_0^y(y).
\end{equation}
In the simple case without coupling, that is,  $W\equiv 0$, the solution stays separated, 
  $\psi(t,x,y) = \varphi^x(t,x)\varphi^y(t,y)$ 
  for all time, where 
  \begin{eqnarray}\nonumber
 i\d_t\varphi^x=H_x\varphi^x \quad;\quad
   \varphi^x_{\mid t=0}= \varphi_0^x,\\\nonumber
   i\d_t\varphi^y=H_y\varphi^y \quad;\quad 
  \varphi^y_{\mid t=0}= \varphi_0^y,
\end{eqnarray}
and this is an exact formula. Here, we aim at investigating the case of an actual
coupling with $\d_x\d_y W\not\equiv 0$ and look for approximate solutions of the form
\[
\psi_{\rm app}(t,x,y) = \varphi^x(t,x)\varphi^y(t,y),
\]
where the individual components satisfy evolution equations that account for the coupling 
between the variables. The main motivation for such approximations is 
dimension reduction, since $\varphi^x(t,x)$ and~$\varphi^y(t,y)$ depend on variables
of lower dimension than the initial $(x,y)$. Of crucial importance is the choice of the approximate 
Hamiltonian
$H_{\rm app} = H_x + H_y + W_{\rm app}(t,x,y)$, 
that governs the approximate dynamics. We consider two different approximate 
coupling potentials: One is the time-dependent Hartree mean-field 
potential, the other one, computationally less demanding, is based on a brute 
force single-point collocation. Time-dependent Hartree methods have been known for a long 
time and have earned the reputation of oversimplifying the dynamics of real molecular systems. 
We emphasize, that our present study does not aim at rejuvenating but at deriving rigorous 
mathematical error estimates, which seem to be missing in the literature.  
Surprisingly, our error analysis provides similar estimates for both methods, the collocation 
and the mean-field approach. 
We investigate the size of the difference between the true and the approximate solution 
in the $L^2$-norm,  
\[
\|\psi(t) - \psi_{\rm app}(t)\|_{L^2} = 
\sqrt{\int_{\R^{n+d}} |\psi(t,x,y)-\psi_{\rm app}(t,x,y)|^2 \,dx dy}
\]
and in Sobolev norms. We present error estimates that explicitly depend on 
derivatives of the coupling potential $W(x,y)$ and on moments of the
approximate solution.
As an additional error measure we also consider the deviation of true and approximate expectation values
$\<\psi(t),A\psi(t)\> - \<\psi_{\rm app}(t),A\psi_{\rm app}(t)\>$, 
for self-adjoint linear operators~$A$. Roughly speaking, the estimates we obtain for observables depend 
on one more derivative of the coupling potential  $W(x,y)$ than the norm estimates. This means that in many 
situations expectation values are more accurately described than the
wave function itself.  Even though rigorous error estimates that quantify 
the decoupling of quantum subsystems in terms of flatness properties of the coupling potential $W(x,y)$ 
are naturally important, our results here seem to be the first ones of their kind.

\subsection{Relation with previous work} 

Interacting quantum systems have traditionally been formulated from the point of view of reduced dynamics theories, based on quantum master equations in a Markovian or non-Markovian setting~\cite{Breuer:02}.
More recently, also tensorized representations of the full quantum system have been considered, 
as for example by matrix product states (MPS) \cite{Schro_Chin,Strat} or within a multiconfiguration time-dependent Hartree (MCTDH) approach~\cite{Burghardt,GattiLasorne,Nest_Meyer,Wang}.  Both wavefunctions (pure states) and density operators (mixed states) can be described in this framework, and wavefunction-based computations can be used to obtain density matrices~\cite{Tamascelli}. In the chemical physics literature, dimension reduction for quantum systems has been proposed in the context of mean-field methods
\cite{Dirac:30,Gerber:82}, and the quantum-classical mean-field Ehrenfest
approach \cite{Delos:72,Billing:83}. Also, quantum-classical formulations have
been derived in Wigner phase space \cite{Martens:96,Kapral:99} and in a
quantum hydrodynamic setting
\cite{Gindensperger:2000,Burghardt:04,Struyve:20}. Our present mathematical formulation
circumvents formal difficulties of these approaches
\cite{Caro:99,Terno:06,Salcedo:12}, by preserving a quantum wavefunction
description for the entire system.
Previous mathematical work we are aware of is concerned with rather specific coupling models, as for example the coupling of Hartree--Fock 
and classical equations in \cite{CanLeB}, or the time-dependent self-consistent field equations 
in \cite{JinSparZho}, or with adiabatic approximations which rely on eigenfunctions for one part of the system, 
see for example \cite{Teu,MarSor}. To the best of our knowledge, the rather general mathematical analysis  
of scale separation in quantum systems we are developing here is new. 

\subsection{Partially flat coupling.}
For a first approximate potential, we consider a brute-force approach, where we collocate 
partially at a single point, for definiteness we choose the origin, and set
\[
W_{\rm bf}(x,y) = W(x,0) + W(0,y) - W(0,0).
\]
In comparison, following the more conventional time-de\-pendent Hartree approach, we set
\[
W_{\rm mf}(t,x,y) = \<W\>_y(t,x) + \<W\>_x(t,y) - \<W\>(t),
\] 
where we perform partial and full averages of the coupling potential, 
\begin{align*}\nonumber
&\<W\>_x = \int_{\R^{d}} W(x,y)\, |\varphi^x(t,x)|^2 \,dx\Big/ \int_{\R^d}|\varphi^x(t,x)|^2\,dx,\\\nonumber
&\<W\>_y = \int_{\R^{d}} W(x,y)\, |\varphi^y(t,y)|^2 \,dy\Big/ \int_{\R^d}|\varphi^y(t,y)|^2\,dy,\\\nonumber
&\<W\> = \int_{\R^{n+d}} W(x,y)\, |\varphi^x(t,x) \varphi^y(t,y)|^2 \,dx dy\Big/
\int_{\R^{n+d}}|\varphi^x(t,x)\varphi^y(t,y)|^2\,dxdy.
\end{align*}
For both approximations, the brute-force and the mean-field
approximation, we derive various types of estimates for the error in
$L^2$-norm. Our key finding is  that both methods come with
error bounds that are qualitatively the same, since they draw from either evaluations or averages of the function
\[
\delta W(x,x',y,y') = W(x,y)-W(x,y')-W(x',y)+W(x',y').
\]
Depending on whether one chooses to control the auxiliary function $\delta W$
in terms of $\nabla_xW$, $\nabla_yW$ or $\nabla_x\nabla_y W$, the estimate
requires a balancing with corresponding moments of the approximate solution.
For example, Proposition~\ref{prop:error-scales} provides $L^2$-norm estimates of the form
\[
 \|\psi(t)-\psi_{\rm app}(t)\|_{L^2}\le
        \begin{cases}
         {\rm const}\,\|\nabla_y W\|_{L^\infty
          }\|\varphi_0^x\|_{L^2_x}\int_0^t\|y\varphi^y(s)\|_{L^2_y}ds,\\
        {\rm const}\,\|\nabla_x \nabla_y W\|_{L^\infty}
          \int_0^t\|x\varphi^x(s)\|_{L^2_x}
        \|y\varphi^y(s)\|_{L^2_y}ds, 
        \end{cases}
      \]
where $\mathrm{const}\in\{1,2,4\}$, depending on whether $\psi_{\rm app}(t)$ results from 
the brute-force or the mean-field approximation. Example~\ref{ex:W} discusses important 
variants of this estimate using different ways of quantifying the flatness of the coupling potential.
Proposition~\ref{prop:error-scales-energy} gives analogous estimates in Sobolev norms. 
In addition, we analyze the
deviation of the true and the approximate expectation values in a similar
vein. For the expectation values, we again obtain qualitatively similar error
estimates for $W_{\rm bf}$ and $W_{\rm mf}$. The upper bounds differ from the norm bounds in so
far as they involve one more derivative of the coupling potential~$W$ and low
order Sobolev norms of the approximate solution, see
Proposition~\ref{prop:observables}. Hence, from the
perspective of approximation accuracy, the brute force and the
mean-field approach differ only slightly. Therefore, other assessment criteria are needed
for explaining the prevalence of the Hartree method in many applications, as we will discuss 
in Section~\ref{subsec:energy}.

\subsection{Dimension reduction via semiclassical analysis.}
In the second part of the paper we turn to a specific case of the previous general class of coupled Hamiltonians 
$H^\eps = H_x + H_y^\eps + W(x,y)$ and consider for one part of the system a semiclassically scaled Schr\"odinger operator
\[
H_y^\eps = -\frac{\eps^2}{2}\Delta_y + V_2(y),\quad \eps>0.
\]
We will discuss in Section~\ref{subsec:scaling} system-bath Hamiltonians that can be 
recast in this semiclassical format.
The initial data are still a product of the form \ref{eq:CI}, but the $y$-factor is chosen as 
\[
\varphi_0^y(y) = \eps^{-d/4}a\(\frac{y-q_0}{\sqrt\eps}\) e^{ip_0\cdot (y-q_0)/\eps},
\]
that is, $\varphi_0^y$ is a semiclassical wave packet with a smooth and rapidly decaying amplitude function 
$a\in{\mathcal S}(\R^d)$, and an arbitrary phase space centre $(q_0,p_0)\in\R^{2d}$. 
We will choose a semiclassical wave packet approximation for $\varphi^y(t,y)$ exploring two different 
choices for the centre $(q(t),p(t))$. As a first option we consider the classical trajectory
\[
\dot q = p\quad,\quad \dot p = -\nabla V_2(q),
\]
and as a second option the corresponding trajectory resulting from the averaged gradient of the potential $V_2$, 
\[
\<\nabla V_2\>(t)= \int_{\R^d} \nabla_y V_2(y)\, |\varphi^y(t,y)|^2\, dy\Big/ \int_{\R^d}|\varphi^y(t,y)|^2\,dy. 
\]
Correspondingly, the approximative factor $\varphi^x(t,x)$ is evolved by the partial Hamiltonian $H_x + W_{\rm eff}$ 
with
\begin{align*}\nonumber
&W_{\rm eff}(t,x) = W(x,q(t)) \quad {\rm or}\\\nonumber
&W_{\rm eff}(t,x) = \int_{\R^d} W(x,y)\, |\varphi^y(t,y)|^2\,dy\Big/ \int_{\R^d}|\varphi^y(t,y)|^2\,dy.
\end{align*}
We obtain error estimates in $L^2$-norm, see Proposition~\ref{prop:error-semi} 
and expectation values, see Proposition~\ref{prop:observables-semi}. 
These estimates are given in terms of the semiclassical parameter~$\eps$ and derivatives 
of the coupling potential. Again, both choices for the effective potential differ only slightly 
in approximation accuracy. Measuring the coupling strength in terms of $\eta = \|\nabla_yW\|_{L^\infty}$, we obtain two-parameter estimates of order $\sqrt\eps+\eta/\sqrt\eps$ in norm,  
while the ones for the expectation values are of order $\eps+\eta$. Hence, again the accuracy 
of quadratic observables is higher than the one for wavefunctions.

\section{Assumptions}\label{sec:hypothesis}
We describe here the mathematical setting that will be ours, and discuss it in the
context of system-bath Hamiltonians \cite{Weiss:12,Breuer:02}. Our Hamiltonian is of the form
\begin{eqnarray}\label{eq:ham}
& H=H_x+H_y +W(x,y),\ {\rm with}\\   
& H_x = -\frac{1}{2}\Delta_x +V_1(x),\quad H_y =-\frac{1}{2}\Delta_y +V_2(y),\nonumber
\end{eqnarray}
where the potentials $V_1(x)$, $V_2(y)$ and the coupling potential $W(x,y)$ are all 
smooth functions, that satisfy growth conditions as given in 
  Assumption~\ref{hyp:potentials}. We will denote $V(x,y) = V_1(x) + V_2(y) + W(x,y)$ and abbreviate the Lebesgue spaces for the different variables $x$, $y$, and $(x,y)$ by 
\[
L^2_x = L^2(\R^n)\quad,\quad L^2_y = L^2(\R^d)\quad, \quad L^2 = L^2(\R^{n+d}).
\] 
The initial data $\psi_0(x,y)$ in \eqref{eq:CI} are products of functions $\varphi_0^x\in L^2_x$ and $\varphi_0^y\in L^2_y$, 
that are square-integrable and typically, Schwartz class, see below.


%
\subsection{Assumptions on regularity and  growth of the
  potentials}\label{sec:assumptions} 

We choose a very classical set of assumptions on the regularity and the growth of the potential, since
our focus is more on finding appropriate ways to approximate the
solution in a standard framework than on treating specific
situations.

\begin{hyp}\label{hyp:potentials}
  All the potentials that we consider are smooth, real-valued, and at
most quadratic in their variables:
     \[
    V_1\in C^\infty(\R^n;\R),\  V_2\in C^\infty(\R^d;\R),\ W\in
    C^\infty(\R^{n+d};\R),
  \]
and, for $\alpha\in \N^n$, $\beta\in \N^d$,
    \begin{eqnarray}\nonumber
 \d_x^\alpha V_1\in L^\infty\ {\rm for}\ |\alpha|\ge 2,
   \quad \d_y^\beta V_2\in L^\infty\ {\rm for} \ |\beta|\ge 2 ,\\\nonumber
    \d_x^\alpha\d_y^\beta W\in L^\infty\ {\rm for}\ |\alpha|+|\beta|\ge 2. 
    \end{eqnarray}
  We also assume that $\nabla_y W\in L^\infty$, but note that this condition can easily be relaxed, see Example~\ref{ex:W}. 
  All the initial date we consider are smooth and rapidly decaying, that is, Schwartz class functions:
  \[
  \varphi^x_0\in{\mathcal
    S}(\R^n;\C),\quad \varphi^y_0\in{\mathcal S}(\R^d;\C)\quad
  ({\rm hence} \ \psi_0\in {\mathcal S}(\R^{n+d};\C)).
  \]
  \end{hyp}
  Under the above assumption, it is well-known that $H_x$, $H_y$ and
  $H$ are essentially self-adjoint on $L^2(\R^N)$, with $N=n,d$ and
  $n+d$, respectively (as a consequence of Faris-Lavine
  Theorem, see e.g. \cite[Theorem~X.38]{ReedSimon2}). 
  \begin{example}
    Since Assumption~\ref{hyp:potentials} involves similar properties
    for $V_1$, $V_2$ or $W$, we present examples for $V_1$ only, which
    can readily be adapted to $V_2$ and $W$. For instance, we can
   consider
    \[
      V_1(x) = \sum_{j=1}^n \omega_j^2 x_j^2 +\sum_{j=1}^n \beta_j
      e^{-\<x,A_j x\>}+ V_{\rm per}(x),
    \]
    with $\omega_j\ge 0$ (possibly anisotropic harmonic potential),
    $\beta_j\in \R$,
    $A_j \in \R^{n\times n}$ positive
    definite symmetric matrices (Gaussian potential), and $ V_{\rm
      per}$ a smooth potential, periodic along some lattice in
    $\R^n$. 
  \end{example} 

The assumptions on the growth and smoothness of the potentials and the regularity of the initial data call for comments.  
  \begin{remark}
Concerning the growth of $V_1$, $V_2$ and $W$, the assumption that they are at most quadratic concerns the behavior at infinity and could be
    relaxed, up to suitable sign assumptions. Local behavior is rather free, for example a local double well is allowed, as long as it is not too confining at infinity. We choose to stick to
    the at most quadratic case, since bounded second order derivatives simplify the presentation.
\end{remark}

\begin{remark}
Concerning the smoothness of our potentials, most of our results still hold assuming only smoothness of $W$, as long as the operators $H_x$ and $H_y$ are essentially self-adjoint on an adequate domain included in $L^2(\R^N)$, with $N=n,d$. For example, $V_1$ and $V_2$ could both present Coulomb singularities, and the results of Proposition~\ref{prop:error-scales} would still hold.  In the semiclassical r\'egime, we can also allow a Coulomb singularity for $V_1$ and prove Proposition~\ref{prop:error-semi} and Proposition~\ref{prop:observables-semi}.
\end{remark}

\begin{remark}
Concerning the smoothness and the decay of the initial data, most of our results still hold, if the initial data are contained in one of the spaces $\Sigma^k(\R^{N})$ containing functions $f$ whose norms 
        \begin{equation}
          \label{eq:Sigmak}
          \| f\|_{\Sigma^k} = \sup_{{z\in \R^N}\atop{|\alpha|+|\beta|\le k}} \| z^{\alpha}\partial_z^{\beta} f\|_{L^2} 
        \end{equation}
 are bounded. Note that $\Sch(\R^N)= \cap_{k\in
      \N}\Sigma^k$. For example, Proposition~\ref{prop:error-scales} still holds for initial data in $\Sigma^1$, 
while Proposition~\ref{prop:error-semi} requires initial data in a semiclassically scaled $\Sigma^3$ space.
\end{remark}

\subsection{System-bath Hamiltonians}\label{sec:repa}
An important class of coupled quantum systems are described by
system-bath Hamiltonians \cite{Weiss:12,Breuer:02}.
\[
H_{\rm sb} = -\frac{1}{2}\Delta_x - \frac{1}{2}\Delta_y + 
V_{\rm s}(x) + V_{\rm b}(y) + V_{\rm sb}(x,y) 
\]
These are naturally given in the format required by \eqref{eq:ham}.
In the present discussion, we specify that the bath
is described by a harmonic oscillator, $V_{\rm b}(y) = \frac12 k_2^0 |y|^2$
(or a set of harmonic oscillators in more than one dimension) and the
system-bath coupling $V_{\rm sb}(x,y) = W(x,y)$ is of cubic form, such that we
obtain in the notation of \eqref{eq:ham}, 
\[
H_x = -\frac12\Delta_x + V_{\rm s}(x)\quad,\quad
H_y = -\frac12\Delta_y + \frac12 k_2^0 |y|^2\quad,\quad
W(x,y) = \frac12\vec\eta\cdot x|y|^2,
\]
where $k_2^0>0$ and $\vec\eta\in\R^d$. The cubic, anharmonic coupling $W(x,y)$
is a non-trivial case, which is employed, e.g., in the description of
vibrational dephasing \cite{Levine:88,Gruebele:04} and Fermi
resonances \cite{Bunker:06}.
It is natural to assume smoothness and subquadratic growth for $V_{\rm s}(x)$. However, 
the coupling potential $W(x,y)$ clearly fails to satisfy the growth condition of 
Assumption~\ref{hyp:potentials}. 
Moreover, it is not clear that in such a framework the total Hamiltonian~$H$ is
essentially self-adjoint. On the other hand, adding a quartic
confining potential,
\[
  H_y = -\frac12\Delta_y + \frac12 k_2^0 |y|^2 +
  \frac{1}{4}k_4^0|y|^4,\quad k_4^0>0,
\]
guarantees that $H$ is essentially self-adjoint. Indeed, Young's inequality for products 
yields
\[
  W(x,y) \ge -\frac{1}{4}\( \frac{1}{c_0} |\eta|^2|x|^2 +c_0|y|^4\),\quad
  \forall c_0>0,
\]
so choosing $c_0 = k_4^0$, we have, for the total potential
$V(x,y)=V_s(x)+V_b(y)+W(x,y)$,
\[
  V(x,y) \ge |V_{\rm s}(x)|  -\frac{|\eta|^2}{8k_4^0}|x|^2\ge -C_1 |x|^2-C_2,
\]
for some constants $C_1,C_2\ge 0$. Hence, the Faris-Lavine Theorem implies that $H_x,H_y$
and $H$ are essentially self-adjoint.
In the following, we will therefore also provide 
slight extensions of our estimates to accommodate this specific, but  
interesting type of coupling (see Remark~\ref{rem:repa}).

\section{Partially flat coupling}\label{sec:flat}

In this section, we present error estimates that reflect partial flatness properties of the coupling 
potential $W(x,y)$ in the sense, that  quantities like $\|\nabla_y
W\|_{L^\infty}$ or $\|\nabla_x\nabla_y
W\|_{L^\infty}$ are small. Depending on the regularity of the initial data, the smallness of these norms
could also be relaxed to the smallness of $\| \<x\>^{-\sigma_x}\langle
y\rangle ^{-\sigma_y} \nabla_y W \|_{L^\infty}$ for some
$\sigma_x,\sigma_y\ge 0$, see Example~\ref{ex:W}. We investigate two approximation strategies, 
one that is based on brute-force collocation, the other one on spatial averaging. 
In each case, we prove that the coupling in $(x,y)$ is negligible
at leading order with respect to $\nabla_yW$. 
Throughout this section, $\psi=\psi(t,x,y)$ denotes the solution to the initial value problem
\eqref{eq:sep-scales}--\eqref{eq:CI}.

\subsection{Brute-force approach}\label{subsec:brute}
We consider the uncoupled system of equations
\begin{equation}\label{eq:varphi}
  \begin{cases}
  i\d_t\varphi^x=H_x\varphi^x + W(x,0)\varphi^x\quad;\quad
   \varphi^x_{\mid t=0}= \varphi_0^x,\\
  i\d_t\varphi^y=H_y\varphi^y + W(0,y)\varphi^y\quad;\quad
  \varphi^y_{\mid t=0}= \varphi_0^y.\end{cases}
\end{equation}
In view of Assumption~\ref{hyp:potentials}, these
  equations have unique solutions $ \varphi^x\in C(\R;L^2_x)$,
  $\varphi^y\in C(\R;L^2_y)$, and higher regularity is propagated, $
  \varphi^x\in C(\R;\Sigma^k_x)$, 
  $\varphi^y\in C(\R;\Sigma^k_y)$, for all $k\in \N$, where we recall
  that $\Sigma^k$ has been defined in \eqref{eq:Sigmak}. The plain product solves 
\[
  i\d_t (\varphi^x\varphi^y) = H (\varphi^x\varphi^y)
  +\(-W(x,y)+W(x,0)+W(0,y)\)(\varphi^x\varphi^y). 
\]
This is not the right approximation, since the residual term is not small: Even if $W$ varies very little in $y$,
then $W(x,y)-W(x,0)-W(0,y)\approx W(x,0) -W(x,0)-W(0,0) = -W(0,0)$. 
This term is removed by considering instead 
\[
\psi_{\rm app}(t,x,y) = e^{itW(0,0)}\varphi^x(t,x)\varphi^y(t,y).
\] 
It satisfies the equation
\[
  i\d_t \psi_{\rm app} = H \psi_{\rm app}
  -\underbrace{\(W(x,y)-W_{\rm app}(x,y)\) \psi_{\rm app}}_{=:\Sigma_\psi} . 
\]
with approximative potential $W_{\rm app}(x,y) = W(x,0)+W(0,y)-W(0,0)$. 
The last term $\Sigma_\psi$ controls the error $\psi-\psi_{\rm app}$, as we will see
more precisely below. Saying that the coupling potential $W$ is flat
in $y$ means that $\nabla_y W$  is
small, and we write
\[
  W(x,y)-W_{\rm app}(x,y)=
  \underbrace{\(W(x,y)-W(x,0)\)}_{\approx y\cdot\nabla_y W(x,0)}-
                               \underbrace{\left(W(0,y)-W(0,0)\right)}_{\approx
                                y\cdot \nabla_y W(0,0)} .
  \]
  This suggests that partial flatness of $W$ implies smallness of the approximation error.  

  \begin{remark}\label{rem:collocation}
  For choosing another collocation point than the origin, one might use the matrix 
  \[M(x,y)= \d_x\d_yW(x,y) =\left( \partial_{x_j}\partial_{y_k} W(x,y)\right)_{1\le j\le n,\,1\le k\le d}.\]
   The Taylor expansion 
   \begin{align}\nonumber
  &  W(x,y) - W_{\rm app}(x,y) = W(x,y)-W(x,0)-W(0,y)+W(0,0) \\\nonumber
&= y\cdot \int_0^1 \left(\d_y W(x,\eta y)-\d_y W(0,\eta y)\right) d\eta\\
&=  \int_0^1\int_0^1 y\cdot \d_x\d_y W(\theta x,\eta y)\, x \ d\eta d\theta.\label{eq:taylor}
\end{align}
  implies for $(x,y) \approx (x_0,y_0)$ that
  \[
    W(x,y)-W(x,y_0)-W(x_0,y)+W(x_0,y_0) \approx    
    (x \cdot M(x_0,y_0)y ),
    \]
   which corresponds to the standard normal mode expansion. Hence, choosing $(x_0,y_0)$ such that the maximal singular value of $M(x_0,y_0)$ is minimized, we minimize the error of the brute-force approach.
  \end{remark}

\subsection{Mean-field approach}\label{subsec:conv}
Instead of pointwise evaluations of the coupling potential, we might also use partial averages 
for an approximation. We consider
\begin{eqnarray}\label{eq:phi}
\begin{cases}
   i\d_t\phi^x=H_x\phi^x + \<W\>_y (t)\,\phi^x
\quad;\quad
   \phi^x_{\mid t=0}= \varphi_0^x,\\
   i\d_t\phi^y=H_y\phi^y +\<W\>_x (t)\,\phi^y\quad;\quad
  \phi^y_{\mid t=0}= \varphi_0^y,\end{cases}
 \end{eqnarray}
where we have denoted
\begin{align*}
\<W\>_y=&  \<W\>_y(t,x) =   \frac{\int W(x,y)|\phi^y(t,y)|^2dy}{\int
   |\phi^y(t,y)|^2dy} =   \frac{\int W(x,y)|\phi^y(t,y)|^2dy}{\int
                 |\varphi_0^y(y)|^2dy},\\
\<W\>_x=&\<W\>_x(t,y)  =   \frac{\int W(x,y)|\phi^x(t,x)|^2dx}{\int
   |\phi^x(t,x)|^2dx} =   \frac{\int W(x,y)|\phi^x(t,x)|^2dx}{\int
                 |\varphi_0^x(x)|^2dx},
\end{align*}
where we have used the fact that the $L^2$-norms of $\phi^x$ and
$\phi^y$ are independent of time, since $W$ is real-valued.
Note that \eqref{eq:phi} is the nonlinear system of equations of the time-dependent Hartree approximation. Contrary to the brute-force approach, $L^2$ regularity does not
  suffice to define partial averages in general.
  In view of Assumption~\ref{hyp:potentials}, a fixed point argument
(very similar to the proof of e.g. \cite[Lemma~13.10]{CaBook2})
  shows that this
  system has a unique solution $ (\phi^x,\phi^y)\in
  C(\R;\Sigma^1_x\times \Sigma^1_y)$, and higher regularity is propagated, $
  \phi^x\in C(\R;\Sigma^k_x)$, 
  $\phi^y\in C(\R;\Sigma^k_y)$, for all $k\ge 2$. 
The approximate solution is then 
\[
\phi_{\rm app} (t,x,y)
=\phi^x(t,x)\,\phi^y(t,y)\, e^{i\int_0^t\<W\>ds},
\]
with the phase given by the full average
\[
\<W\> = \<W\>\!(t) =    
   \frac{\int W(x,y)|\phi^x(t,x)\phi^y(t,y)|^2dxdy}{\int
                 |\varphi_0^x(x)\varphi_0^y(y)|^2dxdy}.
\]
It solves the equation
\[
  i\d_t \phi_{\rm app} = H \phi_{\rm app} - \Sigma_\phi,
  \quad\Sigma_\phi:=  \(W -\<W\>_x -\<W\>_y + \<W\>\)  \phi_{\rm app} .
\]

\begin{remark}
   The correcting phase $e^{i\int_0^t\<W\>ds}$ seems to be
   crucial if we want to compute the   wave function. On the other
   hand, since it is a purely time dependent phase factor, it does not
   affect the usual quadratic observables. The same applies for the phase $e^{itW(0,0)}$ of the brute-force approximation. 
 \end{remark}

\subsection{Error estimates for  wave functions}
\label{sec:error-flat}

We begin with an approximation result at the level of $L^2$-norms
only. For its proof, see~\textcolor{black}{Section~\ref{sec:exproof}.}

\begin{proposition}\label{prop:error-scales}
  Under Assumption~\ref{hyp:potentials}, we have the following error
  estimates:

\smallskip\noindent 
Brute-force approach: for $\psi_{\rm app}(t,x,y)
    =\varphi^x(t,x)\varphi^y(t,y) e^{it W(0,0)}$ defined by \eqref{eq:varphi},
    \[
      \|\psi(t)-\psi_{\rm app}(t)\|_{L^2}\le
      \begin{cases}
         2\|\nabla_y W\|_{L^\infty
          }\|\varphi_0^x\|_{L^2_x}\int_0^t\|y\varphi^y(s)\|_{L^2_y}ds,\\
         \|\nabla_x \nabla_y W\|_{L^\infty}
          \int_0^t\|x\varphi^x(s)\|_{L^2_x}
        \|y\varphi^y(s)\|_{L^2_y}ds. \end{cases}
      \]

\noindent
Mean-field approach: for $\phi_{\rm app}(t,x,y)  =\phi^x(t,x)\phi^y(t,y)
    e^{i\int_0^t\<W\>ds}$ defined by  \eqref{eq:phi},
    \[
      \|\psi(t)-\phi_{\rm app}(t)\|_{L^2}\le
      \begin{cases}
         4\|\nabla_y W\|_{L^\infty  }
         \|\varphi_0^x\|_{L^2_x}\int_0^t\|y\phi^y(s)\|_{L^2_y}ds,\\
         4\|\nabla_x \nabla_y W\|_{L^\infty }
         \int_0^t\|x\phi^x(s)\|_{L^2_x}
        \|y\phi^y(s)\|_{L^2_y}ds. \end{cases}
      \]
\end{proposition}

We see that the smallness of $\|\nabla_y   W\|_{L^\infty}$ controls the error between the exact and the approximate solution in both approaches.

 \begin{example}
An important class of examples consists of those where $W$ is
  slowly varying in~$y$: $W(x,y)={\rm  w}(x,\eta y)$ with $\eta\ll 1$
  and $\rm w$ bounded, as well as its derivatives. In that case 
  $\| \nabla _y W\|_{L^\infty} =\eta \|\nabla_y  {\rm
    w}\|_{L^\infty}$. 
\end{example}
\begin{example}\label{ex:W}
 In the case $W(x,y)=W_1(x)W_2(y)$, the averaged potentials satisfy
  \[
\<W\>_y(t,x) = W_1(x)\<W_2\>_y(t),\quad \<W\>_x(t,y) = \<W_1\>_x(t)W_2(y),
\]
with
\begin{align*}\nonumber
&\<W_2\>=  \<W_2\>_y(t) =   \frac{\int W_2(y)|\phi^y(t,y)|^2dy}{\int
   |\phi^y(t,y)|^2dy} =   \frac{\int W_2(y)|\phi^y(t,y)|^2dy}{\int
                 |\varphi_0^y(y)|^2dy},\\\nonumber
&\<W_1\>=\<W_1\>_x(t)  =   \frac{\int W_1(x)|\phi^x(t,x)|^2dx}{\int
   |\phi^x(t,x)|^2dx} =   \frac{\int W_1(x)|\phi^x(t,x)|^2dx}{\int
                 |\varphi_0^x(x)|^2dx}.
\end{align*}
In this special product case, the crucial source term takes the form
\[
\Sigma_\phi(t) = (W_1-\<W_1\>_x(t))\ (W_2-\<W_2\>_y(t))\ \phi_{\rm app}(t),
\]
and Proposition~\ref{prop:error-scales} can be 
augmented by the gradient-free estimate
\begin{align}\label{eq:gfree}
&\|\psi(t)-\phi_{\rm app}(t)\|_{L^2}\le \\
&\|\varphi_0^x\|_{L^2_x}\|\varphi_0^y\|_{L^2_y}\int_0^t \sqrt{\left(\<W_1^2\>_x(s)-\<W_1\>_x^2(s)\right)
\left(\<W_2^2\>_y(s)-\<W_2\>_y^2(s)\right)} ds.\nonumber
\end{align}
The $L^\infty$-norms, that provide the upper bounds in Proposition~\ref{prop:error-scales}, separate as
  \begin{align*}\nonumber
&    \|\nabla_y W\|_{L^\infty}= \| W_1\|_{L^\infty_x
                  } \|\nabla_y W_2\|_{L^\infty_y},\\\nonumber
    &\|\nabla_x\nabla_y W\|_{L^\infty }= \|\nabla_x W_1\|_{L^\infty_x}
      \| \nabla_y W_2\|_{L^\infty _y},
  \end{align*}
  and it is $\|\nabla_y W_2\|_{L^\infty}$ that controls the estimates.
  Suppose we have $W_2(y)=\eta
|y|^2$ with $\eta$ small: $\nabla W_2$ is not
  bounded, but we can adapt the proof of Proposition~\ref{prop:error-scales} to get
  \[
   \|\psi(t)-\phi_{\rm app}(t)\|_{L^2}\le 16\eta 
 \|W_1\|_{L^\infty_x}\|\varphi_0^x\|_{L^2_x}\int_0^t\||y|^2\phi^y(s)\|_{L^2_y}ds,
\]
that is, the extra power of $y$ is transferred to the $\phi^y$ term.
\end{example} 
\textcolor{black}{
\begin{remark}\label{rem:sigmaxy}
In the spirit of the last observations of Example~\ref{ex:W}, in terms of $\eta:=\| \<x\>^{-\sigma_x}\langle
  y\rangle ^{-\sigma_y}\nabla _y 
W\|_{L^\infty}$, for some $\sigma_x,\sigma_y\ge 0$, we 
get in Proposition~\ref{prop:error-scales}: 
 \[
   \|\psi(t)-\phi_{\rm app}(t)\|_{L^2}\le 
   8 \eta\int_0^t\|\<x\>^{\sigma_x}\phi^x(s)\|_{L^2_x}\| \langle
 y\rangle^{\sigma_y}|y|\phi^y(s)\|_{L^2_y}ds. 
\]
See~\ref{sec:error-scales} for details of the argument. 
\end{remark}
}

\begin{remark}\label{rem:V-momenta}
  If $V_1$
is confining, $V_1(x) \gtrsim |x|^2$ for $|x|\ge R$ (for instance,
$V_1(x)= c|x|^{2k}$, $c>0$ and $k$ a positive integer, a typical case
where $V_1$ may be super-quadratic while $H_x$ and
$H$ remain self-adjoint), then we can
estimate $\|x\phi^x\|_{L^2_x}$ uniformly in time. If $V_1=0$, or more
generally if $V_1(x)\to 0$ as $|x|\to \infty$, we must expect some
linear growth in time
$\|x\phi^x(t)\|_{L^2_x}\lesssim \<t\>$,
and the order of magnitude in $t$ is sharp, corresponding to a
dispersive phenomenon. 
\end{remark}

\begin{remark}\label{rem:repa}
 The framework of a cubic system-bath coupling $W(x,y) = \frac12\vec\eta\cdot x|y|^2$ as described in  Section~\ref{sec:repa} is recovered by taking 
$\sigma_x=\sigma_y=1$ in Example~\ref{ex:W}. In addition, in the presence of a quartic confinement with $k_4^0>0$, in view of
Remark~\ref{rem:V-momenta}, we also know that $\|\langle
 y\rangle|y|\phi^y(t)\|_{L^2_y}$ is bounded uniformly in $t$. 
\end{remark}

Adding control on the gradients of the functions $\varphi^x(t),\varphi^y(t)$ respectively 
$\phi^x(t),\phi^y(t)$, allows 
also error estimates at the level of the kinetic energy. For a proof, 
see~\ref{sec:proof-error-scales-energy}.

 \begin{proposition}\label{prop:error-scales-energy}
      Under Assumption~\ref{hyp:potentials}, there exists a constant
      $C>0$ depending on second order derivatives of the potentials such that we have the following error
  estimates:

\medskip\noindent
Brute-force approach: for $\psi_{\rm app}(t,x,y)
    =\varphi^x(t,x)\varphi^y(t,y) e^{it W(0,0)}$ defined by \eqref{eq:varphi},
    \begin{eqnarray}\nonumber
   \|\nabla_x\psi(t)-\nabla_x\psi_{\rm app}(t)\|_{L^2} +
   \|x\psi(t)-x\psi_{\rm app}(t)\|_{L^2} 
      \le C \|\nabla_x \nabla_y
      W\|_{L^\infty}\times\\\nonumber
\times\int_0^t e^{Cs}
\|y\varphi^y(s)\|_{L^2_y}\(\|x\varphi^x(s)\|_{L^2_x}+
        \|\nabla_x\varphi^x(s)\|_{L^2_x}+\||x|\nabla_x\varphi^x(s)\|_{L^2_x}\)
        ds,
        \end{eqnarray}
        and
        \begin{eqnarray}\nonumber
 \|\nabla_y\psi(t)-\nabla_y\psi_{\rm app}(t)\|_{L^2} +
   \|y\psi(t)-y\psi_{\rm app}(t)\|_{L^2}  \le C \|\nabla_x \nabla_y
      W\|_{L^\infty}\times\\\nonumber
\times\int_0^t e^{Cs}
\|x\varphi^x(s)\|_{L^2_x}\(\|y\varphi^y(s)\|_{L^2_y}+
        \|\nabla_y\varphi^y(s)\|_{L^2_y}+\||y|\nabla_y\varphi^y(s)\|_{L^2_y}\)
        ds.
    \end{eqnarray}
    
\noindent    
Mean-field approach: for $\phi_{\rm app}(t,x,y)  =\phi^x(t,x)\phi^y(t,y)
    e^{i\int_0^t\<W\>ds}$ defined by \eqref{eq:phi}, analogous estimates for $\psi(t)-\phi_{\rm
    app}(t)$ hold.
  \end{proposition}
  
  \begin{remark}
The strategy used to prove Proposition~\ref{prop:error-scales-energy} can be iterated to infer error estimates
in Sobolev spaces of higher order, $H^k(\R^{n+d})$ for $k\ge 2$,
provided that we consider momenta of the same order $k$, which explains the interest in the functional spaces $\Sigma^k$. Error estimates
in such spaces can also be obtained by first proving that $\psi$ and the
approximate solution(s) remain in $\Sigma^k$, 
and
then interpolating with the $L^2$  error estimate from
Proposition~\ref{prop:error-scales}. 
\end{remark}

\subsection{Error estimates for quadratic
  observables}\label{sec:error-quad}
For obtaining quadratic estimates, we 
consider observables such as the energy or the momenta, that is, operators that are differential operators of order at most~$2$ with bounded smooth coefficients.  These 
 differential operators have their domain in $H^2(\R^{n+d})$, as the operator $H$. More generally, we could consider pseudo-differential operators $B={\rm op}(b)$ associated with a smooth real-valued function $b=b(Z)$ with $Z=(z,\zeta)\in \R^{2(n+d)}$, whose action on functions $f\in \mathcal S(\R^{n+d})$ is given by
\[{\rm op}(b) f(z)= (2\pi)^{-(n+d)} \int_{\R^{2(n+d)}} 
b\left(\frac {z+z'}{2} ,\zeta\right) e^{i\zeta\cdot (z-z')} f(z') d\zeta dz'.\]
We assume that $b$ satisfies the H\"ormander condition 
\begin{equation}\label{def:observables}
\forall \alpha,\beta \in \N^{n+d},\;\;\exists C_{\alpha,\beta } >0,\;\;|\partial_z^\beta \partial^\alpha_\zeta b(z,\zeta)|\le C_{\alpha,\beta} \langle  \zeta\rangle^{2-|\alpha|},
\end{equation}
that is, $b$ is a symbol of order $2$, see e.g. \cite[Chapter I.2]{AliGer}.  
We shall also consider observables that  depend only on the variable~$x$ or the variable ~$y$. 
The following estimates are proven in~\ref{sec:proof-observables}.

\begin{proposition}\label{prop:observables}
  Under Assumption~\ref{hyp:potentials}, for  $b\in C^\infty(\R^{n+d})$ satisfying~\ref{def:observables} and $B={\rm op}(b)$,
there exists a constant
      $C_{b}>0$ such that we have the following error
  estimates:
  
  \medskip\noindent
   Brute-force approach: for $\psi_{\rm app}(t,x,y)
    =\varphi^x(t,x)\varphi^y(t,y) e^{it W(0,0)}$, defined by \eqref{eq:varphi}, the error
    \[
    e_\psi(t) = \< \psi(t),B\psi(t)\> - \< \psi_{\rm app}(t),B\psi_{\rm app}(t)\>
    \]
    satisfies
    \[
    \left| e_\psi(t) \right|
     \le C_{b}  \sup_{|\beta|\le N_{n+d}}\|\nabla^\beta M\|_{L^\infty} \| \psi_0\|_{L^2} 
     \left({\mathcal N}(\psi_{\rm app}) + t\|\psi_0\|_{L^2}\right),
    \]
    where $N_{n+d}>0$ depends on $n+d$, $M(x,y) = \partial_x\partial_y W(x,y)$, while 
    \begin{align*}\nonumber
    {\mathcal N}(\psi_{\rm app}) &= \| \varphi^x_0\|_{L^2_x} \int_0^t \| y\varphi^y(s)\|_{L^2_y} ds +
        \| \varphi^y_0\|_{L^2_y} \int_0^t \| x\varphi^x(s)\|_{L^2_x} ds\\\nonumber
        &\quad+\| \varphi^x_0\|_{L^2_x} \int_0^t \| \nabla(y\varphi^y(s))\|_{L^2_y} ds +
        \| \varphi^y_0\|_{L^2_y} \int_0^t \|\nabla(x\varphi^x(s))\|_{L^2_x} ds\\\nonumber
        &\quad+\int_0^t \|x\varphi^x(s)\|_{L^2_x}\|\nabla\varphi^y(s)\|_{L^2_y} ds + 
        \int_0^t \|\nabla\varphi^x(s)\|_{L^2_x}\|y\varphi^y(s)\|_{L^2_y} ds.
    \end{align*}
    
    \noindent
    Mean-field approach: for $\phi_{\rm app}(t,x,y)  =\phi^x(t,x)\phi^y(t,y)
    e^{i\int_0^t\<W\>ds}$ 
  defined by  \eqref{eq:phi}, the error $\langle \psi(t) ,B\psi(t) \rangle   - \langle \phi_{\rm app}(t) ,B\phi_{\rm app}(t) \rangle $ satisfies a similar estimate.
\end{proposition}

\begin{remark} 
The averaging process involved in the action of an observable on a wave function allows to prove 
estimates like the one in Proposition~\ref{prop:observables}, that are more precise than the standard ones 
stemming from norm estimates, 
\[
|e_\psi(t)| \le \|\psi(t)-\psi_{\rm app}(t)\|_{L^2} \left( \|B\psi(t)\|_{L^2} + \|B\psi_{\rm app}(t)\|_{L^2}\right).
\]
\end{remark}

\begin{remark}
We point out  that the error is governed by  derivatives of second
order in~$W$, involving a derivative in the $y$ variable that is
supposed to be small.  Besides, note that the direct use of an estimate on the
wave function itself would have involved $H^2$ norms of  $
\psi_{\rm app}(s) $, while this estimate only requires $H^1$ norms. This first
improvement is due to the averaging process present in  Egorov Theorem. 
 \end{remark}

\subsection{Energy conservation}\label{subsec:energy}

The error estimates of Proposition~\ref{prop:error-scales}, 
Proposition~\ref{prop:error-scales-energy}, and Proposition~\ref{prop:observables}, do not allow to 
distinguish between the brute-force single point collocation and the mean-field Hartree approach. 
However, in computational practice most of the employed methods are of mean-field type. Why? 
Our previous analysis, that specifically addresses the coupling of quantum systems, does not allow for an 
answer, and we resort to a more general point of view. Both approximations, the brute-force and the mean-field one, are norm-conserving. However, the mean-field approach is energy-conserving with the same 
energy as \ref{eq:sep-scales}. At first sight, this is surprising, since the mean-field Hamiltonian~$H_{\rm mf}(t)$
depends on time. In a more general framework, where the time-dependent Hartree approximation 
is considered as application of the time-dependent Dirac--Frenkel variational principle on the 
manifold of product functions, energy conservation is immediate, see 
\cite[\S3.2]{Lu05} or \cite[Chapter~II.1.5]{LubichBlue}.  

\begin{lemma}\label{lem:energy}
Under Assumption~\ref{hyp:potentials} and considering the mean-field approach 
$\phi_{\rm app}(t,x,y)  =\phi^x(t,x)\phi^y(t,y)
    e^{i\int_0^t\<W\>(s)ds}$ defined by \eqref{eq:phi}, we have
\[
\<\phi_{\rm app}(t),H_{\rm mf}(t)\phi_{\rm app}(t)\> =
\<\psi_0,H\psi_0\>\quad\ {\rm for\ all}\ t\ge 0,
\] 
where the mean-field Hamiltonian is given by 
\[
H_{\rm mf}(t) = H_x + H_y + \<W\>_y(t) +  \<W\>_x(t) - \<W\>(t).
\]
\end{lemma}

Below in~\ref{sec:proof-energy} we give an elementary ad-hoc proof of Lemma~\ref{lem:energy}.

\begin{remark}
  In the brute-force case, the approximate Hamiltonian $H_{\rm bf}=H_x + H_y + W(x,0) + W(0,y) - W(0,0)$ is
  time-independent, and we have
  \[
    \<\psi_{\rm app}(t),H_{\rm bf} \psi_{\rm app}(t)\> =  \<\psi_0,H_{\rm bf}\psi_0\>
\quad\ {\rm for\ all}\ t\ge 0.
\]
However this conserved value does not correspond to the exact energy of 
\eqref{eq:sep-scales}, but only to an approximation of it. 
\end{remark}

\section{An exemplary proof} \label{sec:exproof}
Here we discuss our basic proof strategy and apply it for the norm estimate of 
Proposition~\ref{prop:error-scales}. The norm estimates of Example~\ref{ex:W}, Proposition~\ref{prop:error-scales-energy} and Proposition~\ref{prop:error-semi} and the observable estimates of Proposition~\ref{prop:observables} and Proposition~\ref{prop:observables-semi} are carried 
out in~\ref{sec:error-scales} and~\ref{sec:proof-error-semi}.
 
  \begin{lemma}\label{lem:energy-estimate}
    Let $N\ge 1$, $A$ be self-adjoint on $L^2(\R^N)$, and $\psi$
    solution to the Cauchy problem 
    \[
      ih \d_t \psi= A\psi + \Sigma\quad;\quad \psi_{\mid t=0} =\psi_0,
    \]
    where $\psi_0\in L^2(\R^N)$ and $\Sigma\in L^1_{\rm
      loc}(\R^+;L^2(\R^N))$. Then for all $t\ge 0$,
    \[
      \|\psi(t)\|_{L^2(\R^N)} \le \|\psi_0\|_{L^2(\R^N)} +
      \frac{1}{h}\int_0^t \|\Sigma(s)\|_{L^2(\R^N)}ds. 
    \]
  \end{lemma}
  
This standard lemma is our main tool for proving norm estimates. It will be applied with either $h=1$ or $h=\eps$ as parameter. Its proof is given in~\ref{sec:general-lemmas}. Now we present the proof of Proposition~\ref{prop:error-scales}. 

\bigskip\noindent  
{\bf Proof.}
 Denote by $r_\psi = \psi-\psi_{\rm app}$ and $r_\phi=\psi-\phi_{\rm app}$
the errors corresponding to each of the two
approximations presented in Sections~\ref{subsec:brute} and
\ref{subsec:conv}, respectively. They solve
\begin{equation}\label{eq:remainder-flat}
  i\d_t r_\psi=H r_\psi +\Sigma_\psi\quad ;\quad   i\d_t r_\phi=H r_\phi
  +\Sigma_\phi\quad ;\quad r_{\psi\mid t=0}=  r_{\phi\mid t=0}=0.
\end{equation}
We note that both approximations and their components are norm-conserving for all times $t\ge 0$, that is,
\[
    \|\phi^x(t)\|_{L^2_x} = \|\varphi^x(t)\|_{L^2_x}=  \|\varphi_0^y\|_{L^2_x}\ ,\ 
    \|\phi^y(t)\|_{L^2_y} = \|\varphi^y(t)\|_{L^2_y}=  \|\varphi_0^y\|_{L^2_y}.
  \]

\noindent
$\bullet$  In the case of the brute-force approach, we consider the Taylor expansions \eqref{eq:taylor}
and derive the estimates 
  \begin{equation}\label{eq:Sigmapsi}
    \|\Sigma_\psi\|_{L^2}\le
    \begin{cases}
        2\|\nabla_y W\|_{L^\infty}\|y \psi_{\rm app}\|_{L^2} = 2\|\nabla_y
        W\|_{L^\infty} \|\varphi^x\|_{L^2_x} \|y\varphi^y\|_{L^2_y}
        ,\\
        \|\nabla_x\nabla_y  
        W\|_{L^\infty} \|x\varphi^x\|_{L^2_x} \|y\varphi^y\|_{L^2_y}.\end{cases}
  \end{equation}

\noindent
$\bullet$  In the mean-field approach, we note that for $(t,x,y)\in\R\times \R^{n+d}$, 
\begin{eqnarray}\nonumber
  \Bigl(   \int
   |\varphi_0^x(x')\varphi_0^y(y')|^2dx'dy'\Bigr)\(W-\<W\>_x-\<W\>_y+\<W\>\)(t,x,y)\\\nonumber
=
    \int_{\R^{n+d}}
      \underbrace{\(W(x,y)-W(x,y')-W(x',y)+W(x',y')\)}_{=:\delta W(x,x',y,y')}|\phi^x(t,x')\phi^y(t,y')|^2dx'dy'
  \end{eqnarray}
Like we did in the brute-force approach, we may use either of the
estimates
\[
  |\delta W(x,x',y,y')|\le \begin{cases}
      2 |y-y'|\times\|\nabla_yW\|_{L^\infty},\\
      |x-x'| \times|y-y'| \times\|\nabla_x \nabla_yW\|_{L^\infty}.\end{cases}
\]
In the first case, we come up with
 \begin{eqnarray}\nonumber
   \|\Sigma_\phi\|_{L^2}^2 &\le
 4\|\nabla_y W\|_{L^\infty}^2\|\varphi_0^x\|^2_{L^2_x}
                                       \times\\\nonumber
   &\quad\quad \times\int_{\R^{d}}\(\int_{\R^d}|y-y'| |\phi^y(t,y')|^2dy'\)^2|\phi^y(t,y)|^2dy/\|\varphi_0^y\|^4_{L^2_y} .
 \end{eqnarray}
 Now we have
 \begin{align*}\nonumber
   \int_{\R^d}|y-y'| |\phi^y(t,y')|^2dy'& \le \int_{\R^d}\(|y|+|y'|\)
                                          |\phi^y(t,y')|^2dy'\\\nonumber
   &\le |y| \|\varphi_0^y\|_{L^2_y}^2 + \|y\phi ^y(t)\|_{L^2_y}\|\varphi_0^y\|_{L^2_y},
 \end{align*}
 where we have used Cauchy-Schwarz inequality for the last term. We infer
 \[
  \( \int_{\R^d}|y-y'| |\phi^y(t,y')|^2dy'\)^2 \le 2|y|^2
  \|\varphi_0^y\|_{L^2_y}^4 + 2 \|y\phi ^y(t)\|_{L^2_y}^2\|\varphi_0^y\|^2_{L^2_y},
 \]
 where we have used Young inequality $(\alpha+\beta)^2\le
 2(\alpha^2+\beta^2)$, hence
 \[
   \|\Sigma_\phi\|_{L^2}^2\le
   8\|\nabla_y W\|_{L^\infty}^2\|\varphi_0^x\|_{L^2_x}^2\(\int_{\R^d}|y|^2
                                      |\phi^y(t,y)|^2dy+\|y\phi ^y(t)\|_{L^2_y}^2\) ,
\]
 and finally
 \begin{equation}\label{eq:Sigmaphi}
    \|\Sigma_\phi\|_{L^2}\le
    4\|\nabla_y W\|_{L^\infty}\|y\phi^y(t)\|_{L^2_y}
    \|\varphi_0^x\|_{L^2_x}. 
 \end{equation}
In the case of the second type approximation for $\delta W$, we similarly find
\[
    \|\Sigma_\phi\|_{L^2}\le
    4\|\nabla_x \nabla_y W\|_{L^\infty}\|x\phi^x(t)\|_{L^2_x}
    \|y\phi^y(t)\|_{L^2_y}. 
 \]
Proposition~\ref{prop:error-scales} then follows from Lemma~\ref{lem:energy-estimate} with $h=1$. \hfill $\square$

\section{Dimension reduction via semiclassical analysis}\label{sec:semi}

In this section, we consider coupled systems, where one part is governed 
by a semiclassically scaled Hamiltonian, that is, $H_y = H_y^\eps$ with
\[
H^\eps_y = -\frac{\eps^2}{2}\Delta_y + V_2(y).
\]
First we motivate such a partial semiclassical scaling in the context of system-bath 
Hamiltonians and introduce wave packets as natural initial data
for the semiclassical part of the system. We explore partial semiclassical wave packet dynamics guided by classical
trajectories and by trajectories with averaged potentials. Thus, the partially
highly-oscillatory evolution of a PDE in dimension $n+d$ is reduced to a
less-oscillatory PDEs in dimensions $n$, and ODEs in dimension $d$. The
corresponding error estimates in Section~\ref{sec:error} compare the true and the
approximate product solution in norm and with respect to expectation values.

\subsection{Semiclassical scaling}\label{subsec:scaling}
We reconsider the system-bath Hamiltonian with cubic coupling of Section~\ref{sec:repa}, now formulated  
in physical coordinates $(X,Y)$, that is, 
\[
H_{\rm sb} = -\frac{\hbar^2}{2\mu_1}\Delta_{X} + V_{\rm s}(X)  
-\frac{\hbar^2}{2\mu_{2}}\Delta_{Y} + \frac{\mu_2\omega_2^2}{2} |Y|^2 + \frac12\vec\eta\cdot X|Y|^2,
\]
where the coordinates $X$ and $Y$ of the system and the
bath part are prescaled, resulting in the single mass parameters
$\mu_{1},\mu_{2}$ for each subsystem and one single harmonic frequency
$\omega_2$ for the bath (noting that, alternatively, several
harmonic bath frequencies $\omega_{2,j}$ could be introduced, without modifying the
conclusions detailed below).
The corresponding
time-dependent Schr\"odinger equation reads
\[
i\hbar\partial_\tau\Psi(\tau,X,Y) = H_{\rm sb}\Psi(\tau,X,Y).
\]
We perform a local harmonic expansion of the potential 
$V_{\rm s}(X)$ around the origin $X=0$ and assume that it is possible
to determine a dominant frequency~$\omega_1$. We then define the natural length scale of the system as 
\[
\textcolor{black}{a_1} = \sqrt{\frac{\hbar}{\mu_1\omega_1}}.
\]
Rescaling coordinates as $(x,y) = \frac1a (X,Y)$, we obtain
\[
H_{\rm sb} = \hbar\omega_1 \left( -\frac{1}{2}\Delta_x + V_1(x) 
  -\frac{\eps^2}{2}\Delta_y + \frac{1}{2}\frac{\varpi^2}{\eps^2}|y|^2 + \frac12 
  \vec\eta'\cdot x|y|^2\right),
\]
where we have introduced the dimensionless parameters 
\[
\eps = \sqrt{\frac{\mu_1}{\mu_2}}\quad,\quad \varpi = \frac{\omega_2}{\omega_1}\quad,
\]
and denoted $V_1(x) = \frac{1}{\hbar\omega_1}V_s(\textcolor{black}{a_1}x)$ and $\vec\eta' = \frac{\textcolor{black}{a_1}}{\mu_1\omega_1^2}\,\vec\eta$. 
The rescaling of the system potential~$V_s$  and the coupling
vector $\vec\eta$ do not alter their role in the Hamiltonian,
whereas the two dimensionless parameters $\eps$ and $\varpi$ deserve further
attention. We now consider the r\'egime where both the mass ratio $\eps$
between system and bath and the frequency ratio~$\varpi$ between bath and
system are small, that is, where the system is viewed as ``light''
and ``fast'' when compared to the ``heavy'' and ``slow'' bath.

\begin{example}
For the hydrogen molecule H$_2$, where the electrons are considered as
the quantum subsystem while the interatomic vibration is considered as the
classical subsystem, we have $\mu_1 = m_e$ and $\mu_2 = 918.6 m_e$. Further,
the characteristic electronic energy is of the order of $\hbar\omega_1 = 1
E_h$ while the first vibrational level is found at
$\hbar\omega_2 = 0.02005 E_h$. Hence the dimensionless parameters are both
small, $\eps = 0.03299$ and $\varpi = 0.02005$.
\end{example}

\begin{example}
As a second example, we consider coupled molecular vibrations,
exemplified by the H$_2$ molecule in a ``bath'' of rare-gas atoms, here chosen
as krypton (Kr) atoms. The H$_2$ vibration is now considered as a quantum
system interacting with weak intermolecular vibrations. The reduced masses are
given as $\mu_1$(H-H) = 0.5~u = 911.44 $m_e$ (where u refers to atomic mass
units), $\mu_2$ (Kr-Kr) = 41.9 u = $76.379 \times 10^3\ m_e$, and $\mu_3$
(H$_2$-Kr) = 1.953 u = 3560.10 $m_e$. The vibrational quanta associated with
these vibrations are $\hbar\omega_1 \mbox{(H-H)} = 4159.2 \mbox{cm}^{-1} =
0.0189 E_h$, $\hbar\omega_2 \mbox{(Kr-Kr)} = 21.6 \mbox{cm}^{-1} = 9.82 \times
10^{-5}\ E_h$, and $\hbar\omega_3 \mbox{(H$_2$-Kr)} = 26.8 \mbox{cm}^{-1} =
1.22 \times 10^{-4}\ E_h$ (see Refs.\ \cite{Jaeger:16,Wei:05}). The resulting dimensionless mass
ratios are given as $\epsilon_{12} = \sqrt{\mu_1/\mu_2} = 0.109$ and
$\epsilon_{13} = \sqrt{\mu_1/\mu_3} = 0.51$, and the corresponding frequency
ratios are $\varpi_{12} = \omega_2/\omega_1 = 0.005$ and $\varpi_{13} =
\omega_3/\omega_1 = 0.006$. In the case of the H$_2$-Kr relative
motion, note that the frequency ratio $\varpi_{13}$ is indeed small whereas
the mass ratio is $\epsilon_{13} \sim 0.5$; this shows that the
quantum-classical boundary is less clearly defined than in the first example
of coupled electronic-nuclear motions. In such cases, different choices can be made
in defining the quantum-classical partitioning.
\end{example}

In an idealized setting, where $\eps$ is considered 
as a small positive parameter whose size can be arbitrarily small, we would say that
\[
\varpi = \O(\eps)\quad \ as \quad \eps\to0,
\]
and view the system-bath Hamiltonian $H_{\rm sb}$ as an instance of a partially semiclassical operator
\[
H^\eps = -\frac12\Delta_x + V_1(x) -\frac{\eps^2}{2}\Delta_y + V_2(y) + W(x,y),
\]
whose potentials $V_1(x)$ and $V_2(y)$ are independent of the semiclassical parameter~$\eps$ and satisfy the growth conditions
of Assumption~\ref{hyp:potentials}. As emphasized in Section~\ref{sec:repa}, the cubic coupling potential $W(x,y)$ does not satisfy the subquadratic estimate, but can be controlled by additional moments of the 
approximate solution. A corresponding rescaling of time, 
$t = \eps\omega_1\tau$, translates the time-dependent Schr\"odinger equation to its semiclassical counterpart
\begin{equation}
  \label{eq:schrod-semi}
i\eps \partial_t\psi^\eps(t,x,y) = H^\eps\psi^\eps(t,x,y),
\end{equation}
where the physical and the rescaled wave functions are related via 
\[
\psi^\eps(t,x,y) = a^{(n+d)/2}\,\Psi(\tau/(\eps\omega_1),aX,aY).
\]

\begin{remark}
Criteria for justifying a semiclassical description are somewhat versatile in the literature.
Our scaling analysis shows, that for system-bath Hamiltonians with cubic 
coupling the obviously small parameter $\eps$, that describes a ratio of reduced masses, 
has to be complemented by an equally small ratio of frequencies $\varpi$.  
Otherwise, the standard form of an $\eps$-scaled Hamiltonian, as it is typically assumed in the 
mathematical literature, does not seem appropriate.  
\end{remark}

\subsection{Semiclassical initial data and ansatz}\label{subsec:ansatz}

As before, the initial data separate scales,
\begin{equation}
  \label{eq:CI2}
  \psi^\eps(0,x,y) = \varphi_0^x(x)\gg^\eps(y),
\end{equation}
where we now assume that $\gg^\eps$ is a semiclassically
scaled wave packet,
\begin{equation}\label{eq:initial-wp}
  \gg^\eps(y) = \frac{1}{\eps^{d/4}}a\!
\(\frac{y-q_0}{\sqrt\eps}\)e^{ip_0\cdot (y-q_0)/\eps},
\end{equation}
with $(q_0,p_0)\in \R^{2d}$, $a$ smooth and rapidly decreasing, i.e. $a\in\mathcal S(\R^d;\C)= \cap_{k\in\N} \Sigma^k$. In the typical case, where 
the bath is almost structureless (say, near harmonic), the amplitude $a$ 
could be chosen as a complex Gaussian, but not necessarily.
We now seek an approximate solution of the form
  $\psi_{\rm app}^\eps(t,x,y) = \psi_1^\eps(t,x)\psi_2^\eps(t,y)$,
where $\psi_2^\eps$ is a semiclassically scaled wave packet for all time,
\begin{equation}\label{eq:scaling-semi}
\psi^\eps_2(t,y)=  \frac{1}{\eps^{d/4}}u_2\!\(t,\frac{y-q(t)}{\sqrt \eps}\)
  e^{ip(t)\cdot (y-q(t))/\eps +iS(t)/\eps}.
\end{equation}
Here, $(q(t),p(t))\in \R^{2d}$, the phase $S(t)\in \R$, and the amplitude $u_2(t)\in\mathcal S(\R^d,\C)$ must be
determined. 

\begin{remark} \label{rem:BO}We note that our approximation ansatz  
differs from the adiabatic one, 
that would write the full Hamiltonian as $H^\eps = -\frac{\eps^2}{2}\Delta_y + H_{\rm f}(y)$, where
\[
H_{\rm f}(y) =  -\frac12\Delta_x + V_1(x) + V_2(y) + W(x,y)
\]
is an operator, that parametrically depends on the ``slow'' variable $y$ and acts on the 
``fast'' degrees of freedom $x$. From the adiabatic point of view, one would then construct an 
approximate solution as $\psi_{\rm bo}^\eps(t,x,y) = \Phi(x,y)\psi_2^\eps(t,y)$, where $\Phi(x,y)$ 
is an eigenfunction of the operator $H_{\rm f}(y)$; here, the subscript
``bo'' stands for Born-Oppenheimer. 
The result of Corollary~\ref{cor:BO} emphasizes the difference between these two points of view.
\end{remark}

We denote by
\begin{equation}\label{def:uapp}
  u_{\rm app}^\eps (t,x,z) = \psi_1^\eps(t,x)u_2(t,z)\quad {\rm with}\quad z=\frac{y-q(t)}{\sqrt \eps}
\end{equation}
the part of the approximate solution that just contains the amplitude. With this notation, 
\[
\psi^\eps_{\rm app}(t,x,y) = 
\frac{1}{\eps^{d/4}} u_{\rm app}^\eps(t,x,z)
 e^{ip(t)\cdot z/\sqrt\eps + i S(t)/\eps}\Big|_{z=\frac{y-q(t)}{\sqrt\eps}}.
\]
The analysis  developed in the next two sections allows to derive two different 
approximations, based on ordinary differential equations governing the semiclassical wave packet part,  which are justified in Section~\ref{sec:error} (see Proposition~\ref{prop:error-semi}). 

\subsection{Approximation by partial Taylor expansion}\label{sec:taylor-semi}
Plugging the  expression of~$\psi_{\rm app}^\eps(t,x,y)$
into \eqref{eq:schrod-semi} and writing $y=q(t)+z\sqrt\eps$ in combination with the Taylor expansions
\begin{eqnarray}\nonumber
  V_2(y) = V_2(q(t)+z\sqrt\eps) = V_2(q(t))+\sqrt\eps z\cdot \nabla
  V_2(q(t)) + \frac{\eps}{2}\<z,\nabla^2 V_2(q(t))z\>
           +\O(\eps^{3/2}),\\\nonumber
  W(x,y)=W(x,q(t)+z\sqrt\eps)=  W(x,q(t))+\sqrt\eps z\cdot
          \nabla_yW(x,q(t))+\O\(\eps\),
\end{eqnarray}
we find:
\begin{align*}
  i\eps &\d_t \psi_{\rm app}^\eps   +\frac{1}{2}\Delta_x
         \psi_{\rm app}^\eps +\frac{\eps^2}{2}\Delta_y\psi^\eps_{\rm app}-V(x,y)\psi_{\rm app}^\eps=
         \\
 &  \eps^{-d/4} e^{ip(t)\cdot z/\sqrt\eps +iS(t)/\eps}\times\\
&\Big(  \Big(p\cdot \dot q -\dot S
        -\frac{|p|^2}{2}-V_2(q)-V_1(x)- W(x,q)\Big)u_{\rm app}^\eps \\
&+\sqrt\eps \(-i\dot q \cdot\nabla_z u_{\rm app} ^\eps-\dot p \cdot z\,
  u_{\rm app}^\eps+ip\cdot\nabla_z u_{\rm app}^\eps-z\cdot \nabla V_2(q)u_{\rm app}^\eps- z\cdot \nabla_y W(x,q)u_{\rm app}^\eps \)\\
 &   +\eps\( i\d_t u_{\rm app} ^\eps +\frac{\aaa}{2\eps}\Delta_x u_{\rm app}^\eps 
         +\frac{1}{2}\Delta_z u_{\rm app}^\eps
  -\frac{1}{2}\<z,\nabla^2 V_2\(q\)z\>u_{\rm app}^\eps \)\\
&+\O(\eps^{3/2})\Big),
\end{align*}
where the argument of $u_{\rm app}^\eps$ and its derivatives are taken
in $z=\frac{y-q(t)}{\sqrt\eps}$.  
To cancel the first four terms in the $\sqrt\eps$ line, it is natural to require
\begin{equation}\label{eq:trj}
  \dot q = p,\quad q(0)=q_0,\quad \dot p = -\nabla V_2(q),\quad
    p(0)=p_0.
\end{equation}
Now cancelling the first four terms in the first line of the right
hand side yields
\begin{equation}\label{eq:action}
 \dot S(t) = \frac{|p(\textcolor{black}{t})|^2}{2}-V_2(q(\textcolor{black}{t})),\quad S(0) =0.
\end{equation}
In other words, $(q(t),p(t))$ is the classical trajectory in $y$, and
$S(t)$ is the associated classical action. At this stage, we note that
the term $z\cdot \nabla_y W(x,q)u_{\rm app}^\eps $ is not compatible with decoupling the variables $x$ and $z$ (or equivalently, $x$ and
$y$). 
Using that $\|\nabla_y W\|_{L^\infty}$ is assumed to be small, 
the above computation becomes
\begin{align*}
  i\eps &\d_t \psi_{\rm app} ^\eps  +\frac{1}{2}\Delta_x
         \psi_{\rm app} ^\eps+\frac{\eps^2}{2}\Delta_y\psi_{\rm app} ^\eps-V(x,y)\psi_{\rm app} ^\eps= 
  \eps^{-d/4} e^{ip(t)\cdot z/\sqrt\eps +iS(t)/\eps}\times\\
&\Big(  i\eps\d_t u_{\rm app} ^\eps+\frac{1}{2}\Delta_x u_{\rm app}^\eps
         +\frac{\eps}{2}\Delta_z u_{\rm app} ^\eps
  -\(\frac{\eps}{2}\<z,\nabla^2 V_2\(q\)z\>+V_1(x)+W(x,q)\)u_{\rm app}^\eps\\
&+\O\(\eps^{3/2}+ \sqrt\eps\,  \|\nabla_y W\|_{L^\infty}\)\Big).
\end{align*}
In view of~\eqref{def:uapp}, we  set 
\begin{eqnarray} \label{eq:psi1}
&  i\eps\d_t \psi_1 ^\eps+ \frac{1}{2}\Delta_x \psi_1 ^\eps=
  \(V_1(x)+W(x,q)\)\psi_1^\eps\quad;\quad \psi_{1\mid t=0} ^\eps= \varphi_0^x\\
\label{eq:u2}
&    i\d_t u_2 + \frac{1}{2}\Delta_z u_2 = \frac{1}{2}\<z,\nabla^2
  V_2\(q\)z\>u_2\quad ;\quad u_{2\mid t=0} = a.
\end{eqnarray}
Equation~\eqref{eq:u2} is a Schr\"odinger equation
  with a time-dependent harmonic potential: it has a unique solution
  in $L^2$ as soon as $a\in L^2(\R^d)$. In addition, since $a\in
  \Sigma^k$ for all $k\in \N$, $u_2\in C(\R;\Sigma^k_z)$ for all $k\in \N$.
The validity of this approximation is stated in
Proposition~\ref{prop:error-semi} below. Note that
  if $a$ is a Gaussian state, then $u_2$ too and its (time-dependent)
  parameters -- width matrix and  centre point -- can be computed by solving ODEs
  (see e.g. ~\cite{LubichBlue,CaBook2,CoRoBook,LasserLubich} and
  references therein).

\subsection{Approximation by partial averaging}\label{sec:average-semi}
Following e.g. \cite{CoalsonKarplus,RoeBurg} or \cite[Section~2]{LasserLubich}, we write
\begin{align*}
& V_2(y) = V_2(q(t)+z\sqrt\eps) = \langle V_2\rangle_y(t)+ \sqrt\eps z\cdot\langle \nabla V_2\rangle_y(t) + 
\frac\eps2 \ z\cdot \langle \nabla^2 V_2\rangle_y(t) z + {\rm v}_1,\\
 &W(x,y) = W(x,q(t)+z\sqrt\eps) = \langle W(x,\cdot)\rangle_y(t)+ {\rm v}_2,
\end{align*}
where the averages are with respect to $|\psi_2^\eps(t,y)|^2dy$. For example,
\begin{align}\label{whyu2eps}
\langle \nabla^2 V_2\rangle_y(t) &=
\frac{\int\nabla^2 V_2(y)|\psi_2(t,y)|^2dy}{\int |\psi_2(t,y)|^2dy}\\\nonumber
&=
\frac{1}{\|a\|_{L^2(\R^d)}^2}\int_{\R^d} \nabla ^2V_2(q(t)+\sqrt\eps z)\ \big|u_2(t,z)\big|^2 dz,\\\nonumber
\langle W(x,\cdot)\rangle_y(t) &= \frac{1}{\|a\|_{L^2(\R^d)}^2}\int_{\R^d} W(x,q(t)+\sqrt\eps z) \ \big|u_2(t,z)\big|^2 dz,
\end{align}
where we anticipate the fact that the $L^2_y$-norm of $\psi_2^\eps(t)$ is independent of time.
We almost literally repeat the previous argument and find that
\begin{eqnarray}\nonumber
  i\eps \d_t \psi_{\rm app} ^\eps  +\frac{1}{2}\Delta_x
         \psi_{\rm app} ^\eps+\frac{\eps^2}{2}\Delta_y\psi_{\rm app} ^\eps -V(x,y)\psi_{\rm app} ^\eps= 
  \eps^{-d/4} e^{ip(t)\cdot z/\sqrt\eps +iS(t)/\eps}\times\\\nonumber
  \Big(  i\eps\d_t u_{\rm app} ^\eps+\frac{1}{2}\Delta_x u_{\rm app}^\eps
         +\frac{\eps}{2}\Delta_z u_{\rm app} ^\eps
  -\(\frac{\eps}{2}\ z\cdot\<\nabla^2 V_2\>_y z +V_1(x)+\langle W(x,\cdot)\rangle_y\)u_{\rm app}^\eps\\\nonumber
\ + \,\tilde r_1^\eps+\tilde r_2^\eps\Big)\Big|_{z=\frac{y-q(t)}{\sqrt{\eps}}}
\end{eqnarray}
with 
 $\tilde r_j^\eps={\rm v}_j u_{\rm app}^\eps$, $j=1,2$, and $z$ is taken as $z=(y-q(t))/\sqrt\eps$. 
 The parameters satisfy the equations of motion
 \[
  \dot q = p,\ q(0)=q_0,\ 
  \dot p = -\langle\nabla V_2\rangle_y(t),\ p(0)=p_0,\ 
 \dot S(t) = \frac{|p(t)|^2}{2}- \langle V_2\rangle_y(t),\ S(0) = 0.
\]
We see that we can now define the approximate solution by:
\begin{equation}\label{eq:psi1av}
  i\eps\d_t \psi_1 ^\eps+ \frac{1}{2}\Delta_x \psi_1^\eps =
  \(V_1(x)+\<W(x,\cdot)\>_y(t)\)\psi_1^\eps\quad;\quad \psi_{1\mid t=0}^\eps = \varphi_0^x,
\end{equation}
\[
  i\d_t u_2 + \frac{1}{2}\Delta_z u_2 = \frac{1}{2}\ z\cdot\<\nabla^2
  V_2\>_y\!(t)\,z \ u_2\quad ;\quad u_{2\mid t=0} = a. 
\]
  Since the matrix $\langle\nabla^2
  V_2\rangle_y\!(t)$ is real-valued, we infer that the $L^2_z$-norm of
  $u_2(t)$ is independent of time, hence $\|\psi_2(t)\|_{L^2_y} =
  \|u_2(t)\|_{L^2_z}=\|a\|_{L^2}$. The equation in $u_2$ is now nonlinear, and can be
  solved in $\Sigma^1$, since $\nabla V_2$ is at most linear in its 
  argument: $u_2\in C(\R;\Sigma^1_z)$, and higher $\Sigma^k$ regularity
  is propagated. Here again, if $a$ is a Gaussian, then so is $u_2$ and its width and centre
can be computed by solving ODEs (see~\cite{LubichBlue,CaBook2}). Note
also that, differently from the previous setting, $u_2$ is now $\eps$-dependent via the quantity $\langle\nabla^2 V_2\rangle_y(t)$ (see~\eqref{whyu2eps}).
However, this dependence is very weak since a Taylor expansion
in~\eqref{whyu2eps} shows that $u_2$ is close in any $\Sigma^k$ norm from the
solution of the equation~\eqref{eq:u2}.
  For this reason, we do not keep memory of this $\eps$-dependence and
  write~$u_2$. By contrast, the $\eps$-dependence of $\psi_1^\eps$ is
  strong since it involves oscillatory features in time.

\subsection{The approximation results}\label{sec:error}
The main outcome of the approximations can be stated as follows,
and is proved in~\ref{subsec:error-semi}:

\begin{proposition}\label{prop:error-semi}
  Let $\psi^\eps $ be the solution to
  \eqref{eq:schrod-semi}--\eqref{eq:CI2}, with $\gg^\eps$ given by
  \eqref{eq:initial-wp}. Then with $\psi_{\rm app}^\eps$ given either like
  in Section~\ref{sec:taylor-semi} or like in
  Section~\ref{sec:average-semi}, there exist constants $K_0,K_1$
  independent of $\eps$ such that for all $t\ge 0$,
  \[
    \|\psi^\eps (t)- \psi_{\rm app}^\eps(t)\|_{L^2}\le K_0\( \sqrt\eps +
    \frac{\|\nabla_y W\|_{L^\infty}}{\sqrt\eps}  \) e^{K_1 t}.
 \]
\end{proposition}

\begin{corollary}
Assume $\eta:=\|\nabla_y W\|_{L^\infty}\ll \sqrt\eps$, then for all $T>0$,
\[\sup_{t\in[0,T]}  \|\psi^\eps(t)- \psi_{\rm app}^\eps(t)\|_{L^2}= \O\left( \sqrt\eps + \frac{\eta}{\sqrt\eps}\right).\]
\end{corollary}

\begin{remark}\label{rem:Sobolev-semi}
Using the same techniques as in~\ref{sec:proof-error-scales-energy}, one can prove estimates on higher regularity norms, using $\eps$-derivatives in $y$ and standard ones in $x$. For example, 
if $a\in\Sigma^4$, then there exists $K_0$, $K_1$ independent of $\eps$ such that 
\begin{align*}
& \| \eps \nabla_y \psi^\eps(t)-\eps\nabla_y  \psi^\eps_{\rm app} (t)\|_{L^2}+ \| y \psi^\eps(t)-y  \psi^\eps_{\rm app} (t)\|_{L^2}\\
&\quad\le 
 K_0\left( 
   \sqrt \eps \int_0^t e^{K_1 s} \| u_2(s)\|_{\Sigma^4} ds 
   +   \frac{\|\nabla_y W\|_{L^\infty}}{\sqrt\eps}  \int_0^t e^{K_1s} \| u_2(s)\|_{\Sigma^2} ds\right).
\end{align*}
We refer to \cite{CaFe11} (see also \cite[Chapter~12]{CaBook2}) for
  more detailed computations.
\end{remark}

Note that, in both approximations,  the evolution of $u_2$ corresponds to the standard quadratic
approximation. In particular, if $a$ is Gaussian, then $u_2$ is
Gaussian at all time, and solving the equation in $u_2$ amounts to
solving ordinary differential equations. However, the equation~\eqref{eq:psi1} solved by $\psi_1(t)$ is still quantum, such that a reduction of the total space dimension of the quantum system has been made from $n+d$ to $n$. 

\bigskip
Let us now discuss the approximation of observables that we choose as acting only in 
 the variable $y$. Due to the presence of the small parameter $\eps$, we choose semiclassical observables and associate with $b\in C^\infty_c(\R^{2d})$ ($b$ smooth and compactly supported) the operator 
${\rm op}_\eps(b)$ whose action on functions $f\in \mathcal S(\R^d)$ is given by
\[{\rm op}_\eps(b) f(y)= (2\pi\eps)^{-d} \int_{\R^{2d}} b\left(\frac {y+y'}{2} , \xi\right)
e^{i\xi\cdot (y-y')/\eps} f(y') \, d\xi dy'.\]
As before, the error estimate is better for quadratic observables than for the wave functions. 
More specifically, the following result, that is proved in \ref{sec:error-quad-semi}, improves the error estimate from Proposition~\ref{prop:error-semi} by a factor $\sqrt\eps$.

\begin{proposition}\label{prop:observables-semi}
  Let $\psi^\eps$ be the solution to
  \eqref{eq:schrod-semi}--\eqref{eq:CI2}, with $\gg^\eps$ given by
  \eqref{eq:initial-wp}. Then with $b\in C^\infty_c(\R^{2d})$ and $\psi_{\rm app}^\eps$ given either like
  in Section~\ref{sec:taylor-semi} or in
  Section~\ref{sec:average-semi}, there exist a constant $K$
  independent of $\eps$ such that for all $t\ge 0$,
  \[
    \left|\langle \psi ^\eps(t),{\rm op}_\eps (b) \psi^\eps(t)\rangle -\langle  \psi_{\rm app}^\eps(t),{\rm op}_\eps(b)  \psi_{\rm app}^\eps(t) \rangle \right|
    \le K\, t \( \eps +
    \|\nabla_y W\|_{L^\infty}\).
 \]
\end{proposition} 

\begin{remark}
Of course, we could have considered a mixed setting consisting of pseudodifferential operators as in Section~\ref{sec:error-quad} in the variable~$x$, and semiclassical as above in the variable~$y$. One would then obtain estimates mixing those of Proposition~\ref{prop:observables} and Proposition~\ref{prop:observables-semi}. 
\end{remark}


\section{\textcolor{black}{A numerical example}}\label{sec:num}
\textcolor{black}{
\begin{table}\centering
\begin{tabular}{l|l|l}
model variation & $\varpi = \omega_2/\omega_1$ & $\eta$ (coupling) \\\hline
blue & $\varpi_{\rm ref}$ & $\eta_{\rm ref}$ \\
red & $\frac14\varpi_{\rm ref}$ & $\eta_{\rm ref}$\\
grey & $\varpi_{\rm ref}$ & $\frac98 \eta_{\rm ref}$\\
yellow & $\varpi_{\rm ref}$ & $\frac{3}{4}\eta_{\rm ref}$ 
\end{tabular}
\caption{\label{tab}Parameters defining the four numerically simulated variations of the system-bath Hamiltonian~\eqref{eq:ham}. The blue model uses 
the coupling parameter $\eta_{\rm ref} = -k_2^0/(3a_1\ell)$ and 
the frequency ratio $\varpi_{\rm ref}=1/100$. The red model varies the frequency ratio, the grey and yellow models the coupling strength.}
\end{table}
For an illustrative numerical application, we consider a system-bath type Hamiltonian with cubic coupling $W(x,y)=\frac12 \eta xy^2$, as developed in Section~\ref{sec:repa}, 
\[
i\partial_t\psi(t,x,y) = H_{\rm sb}\psi(t,x,y)\quad;\quad \psi(0,x,y) = \varphi^x_0(x)\varphi^y_0(y),
\]
in dimension $d=n=1$,
\begin{equation}\label{eq:ham}
H_{\rm sb} = -\frac{1}{2\mu_1}\frac{\d^2}{\d x^2} -
\frac{1}{2\mu_2}\frac{\d^2}{\d y^2} + V_{\rm s}(x) + V_{\rm b}(y) + 
\frac{1}{2}\eta xy^2.
\end{equation}
The mass ratio between the system and the bath is moderately small, $\mu_1/\mu_2 = 0.25$.  
The system potential is a quartic double well, while the bath potential is harmonic, 
\[
V_{\rm s}(x) = \frac12 x^2\left(\frac{x}{2\ell}-1\right)^2, \quad V_{\rm b}(y) = \frac12 k_2^0 y^2.
\]
The length scale $\ell=4a_1$ of the double well is a multiple of the system's natural harmonic 
unit~$a_1$. 
The initial data $\psi_0(x,y)=\varphi^x_0(x)\varphi^y_0(y)$ are the ground state of the bivariate harmonic oscillator, that results from the harmonic approximation of $V_{\rm s}(x) + V_{\rm b}(y)$ around the left  
well $(x,y)=(0,0)$. The coupling constant 
$\eta<0$ is negative to ensure that the total Hamiltonian's ground state is localised in the right well $(x,y)=(2\ell,0)$, providing a setup with pronounced non-equilibrium dynamics.
For such a system-bath model, the gradient-free error estimate \eqref{eq:gfree} of Example~\ref{ex:W} is given by
\begin{align}\label{eq:normest}
&\|\psi(t)-\phi_{\rm app}(t)\|_{L^2}\le \\\nonumber
&\frac{1}{2}|\eta| \,\|\varphi_0^x\|_{L^2_x}\|\varphi_0^y\|_{L^2_y}\int_0^t \sqrt{\left(\<x^2\>_x(s)-\<x\>_x^2(s)\right)
\left(\<y^4\>_y(s)-\<y^2\>_y^2(s)\right)} ds.
\end{align}
\begin{figure}
\begin{subfigure}{0.49\textwidth}
\includegraphics[width=\linewidth]{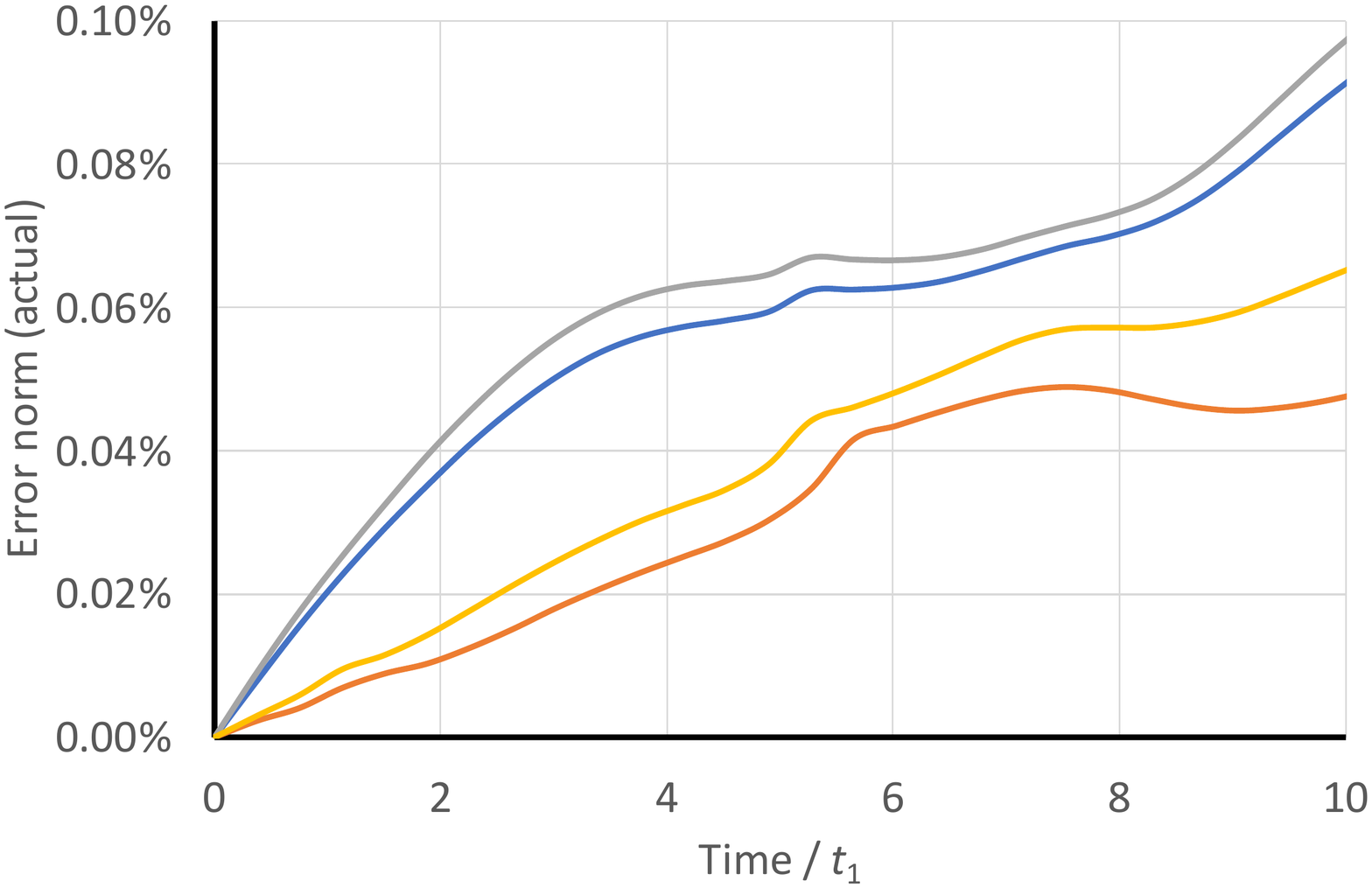}
\caption{\label{figa}Actual numerical error}
\end{subfigure}
\hspace*{\fill}
\begin{subfigure}{0.49\textwidth}
\includegraphics[width=\linewidth]{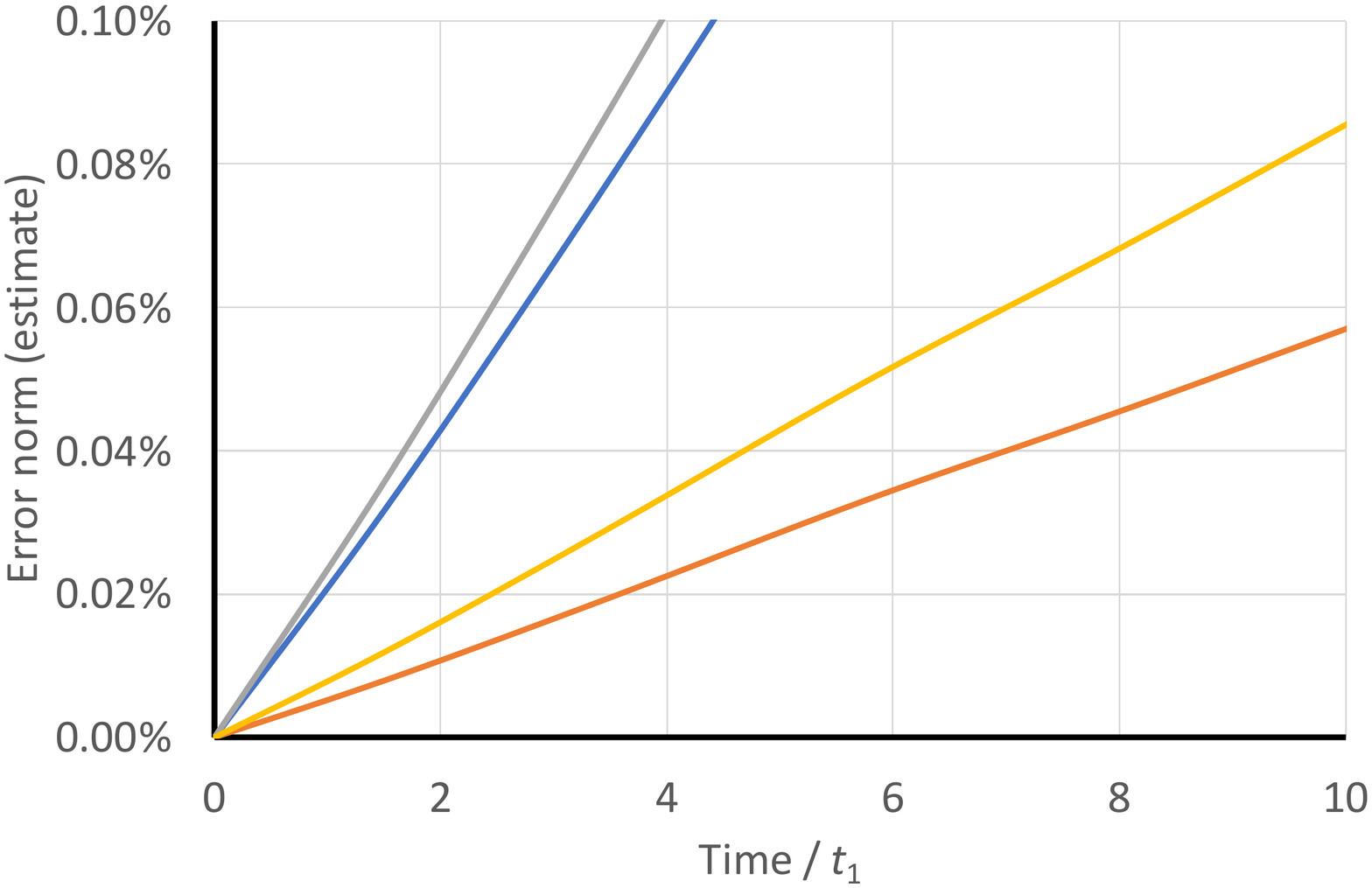}
\caption{\label{figb}Theoretical error estimate}
\end{subfigure}
\caption{\label{fig}Error of the mean-field approximation for a family of system-bath models with cubic coupling $W(x,y)=\frac12\eta xy^2$ as a function of time. The left plot shows the time evolution of the norm of the error for four different variations of the parameter values as given in Table~\ref{tab}.  
The right plot shows the corresponding upper bound of the error estimate \eqref{eq:normest}.}
\end{figure}
Figure~\ref{fig} presents the results from the following numerical experiment: We identify a 
frequency ratio $\varpi_{\rm ref}=1/100$ between bath and system and a coupling parameter 
$\eta_{\rm ref} = -k_2^0/(3a_1\ell)$ as generating ``reference'' system-bath dynamics. For this parameter choice, 
the mean-field approximation, when compared with a numerically converged MCTDH approximation, 
results in roughly a $0.1\%$ error after $10$ units of the natural 
harmonic time scale $t_1= 1/\omega_1$ of the system (blue curve), see Figure~\ref{figa}. Hence, the Hartree approximation is excellent 
on the time scale under study.
Decreasing the frequency ratio by a factor of four, roughly halves the error (red curve). 
And, as expected, an increase in the coupling strength also increases the error (grey curve), while decreasing 
the coupling also decreases the error (yellow curve). The corresponding plot of 
Figure~\ref{figb} illustrates that the upper bound of the theoretical error estimate correctly captures 
the initial slope of all four error curves, while slightly over-estimating the actual error as time evolves. 
A more detailed assessment of the error estimate \eqref{eq:normest}, in particular of its long-time behaviour (up to 150 ps), when the mean-field approximation goes up to errors of the order of $1\%$,  
and a more complete screening of physically relevant parameter r\'egimes are work in progress for a numerical companion paper to the present theoretical study.}

\section{Conclusion and outlook}
We have presented quantitative error bounds for the approximation of quantum dynamical wave functions in product form. For both considered approaches, a brute-force single point collocation and the conventional mean-field Hartree approximation, we have obtained similar error estimates in $L^2$-norm (Proposition~\ref{prop:error-scales},  Example~\ref{ex:W}), in $H^1$-norm (Proposition~\ref{prop:error-scales-energy}), and for quadratic observables (Proposition~\ref{prop:observables}). To our knowledge, 
such general estimates, that quantify decoupling in terms of flatness properties of the coupling potential, are new. 
The corresponding analysis for semiclassical subsystems (Proposition~\ref{prop:error-semi},  Proposition~\ref{prop:observables-semi}) confirms the more general finding, that error estimates for 
quadratic observables provide smaller bounds than related norm estimates. 
The single product analysis, as presented here, provides a 
well-posed starting point for the investigation of more elaborate approximation methods.  
If the initial data satisfy
 \[
     \psi(0,x,y)=\psi_0(x,y) = \sum_{j=1}^J \varphi_{0j}^x(x)\varphi_{0j}^y(y),
   \]
  then we may invoke the linearity of \eqref{eq:sep-scales} to write
  $\psi(t,x,y) =  \sum_{j=1}^J  \psi_j(t,x,y)$, 
   where each $\psi_j$ solves
   $   i\d_t \psi_j = H \psi_j$, with $ \psi_{j\mid t=0} =
     \varphi_{0j}^x\otimes\varphi_{0j}^y$.
We approximate each $\psi_j\approx\psi_{j,{\rm app}}$ individually in one of the ways discussed in the present paper and use the triangle inequality for
\[
  \big\| \psi(t) - \sum_{j=1}^J\psi_{j,{\rm app}}(t)\big\|\le
  \sum_{j=1}^J \|\psi_j(t) - \psi_{j,{\rm app}}(t)\|,
\]
where the norm can be an $L^2$ or an energy norm, for
instance. However, working on each $\psi_j$ instead of 
$\psi$ directly, seems to prevent control of the limit $J\to\infty$. Multi-configuration methods therefore use ansatz functions of the form 
  \[
     \psi_{\rm app}(t,x,y) =  \sum_{j,\ell} a_{j\ell}(t) \, 
     \varphi_{j}^{(x)}(t,x)\, \varphi_{\ell}^{(y)}(t,y),
   \]     
where the families $(\varphi^{(x)}_j(t))_{j\ge1}$ and $(\varphi^{(y)}_\ell(t))_{\ell\ge1}$ satisfy orthonormality or rank conditions, while gauge constraints lift redundancies for the coefficients $a_{k\ell}(t)\in\C$.
We view our contributions here as an important first step for a systematic assessment 
of such multi-configuration approximations in the context of coupled quantum systems.
A numerical companion paper, that explores the dynamics of system-bath Hamiltonians with 
cubic coupling \textcolor{black}{on multiple time-scales with respect to various parameter r\'egimes}, is currently in preparation.

\subsection*{Acknowledgements}
Adhering to the prevalent convention in mathematics, 
we chose an alphabetical ordering of authors. We thank Lo\"ic Joubert-Doriol, G\'erard Parlant,
Yohann Scribano, and Christof Sparber for stimulating discussions in the context of this paper. We acknowledge support from the CNRS 80$|$Prime 
project {\it AlgDynQua} and the visiting professorship program of the I-Site Future.

\appendix
\section{General estimation lemmas}\label{sec:general-lemmas}
Here we provide the proof of the standard energy estimate, Lemma~\ref{lem:energy-estimate}.

\begin{proof}
In view of the self-adjointness of $A$, we have
\begin{align*}
\|\psi(t)\|\ \frac{d}{dt}\|\psi(t)\| &= \frac12\frac{d}{dt} \langle \psi(t),\psi(t)\rangle
= \Re\<\psi(t),\frac{1}{ih} (A\psi(t)+\Sigma(t))\>\\
&= \frac{1}{h}\Im\langle\psi(t),\Sigma(t)\rangle, 
\end{align*}
and therefore, by the Cauchy--Schwarz inequality, 
$\frac{d}{dt}\|\psi(t)\|  \le \frac{1}{h}\|\Sigma(t)\|$.
Integrating in time, we obtain
\[
\|\psi(t)\| = \|\psi_0\| + \int_0^t \frac{d}{ds}\|\psi(s)\| ds \le 
\|\psi_0\| + \frac{1}{h}\int_0^t \|\Sigma(s)\| ds.
\]
\end{proof}

In the context of observables, refined error estimates will follow from
the application of the following lemma. 

\begin{lemma}\label{lem:quad-estimate}
Let $N\ge 1$, $A_1$, $A_2$, $B$ be self-adjoint on $L^2(\R^N)$, and $\psi^{(1)}, \psi^{(2)}$, 
    solutions to the homogeneous Cauchy problems 
    \[
      ih \d_t \psi^{(j)}= A_j\psi^{(j)} \quad;\quad \psi^{(j)}_{\mid t=0} =\psi_0,
    \]
    where $\psi_0\in L^2(\R^N)$. Then, for all $t\ge 0$ 
    \[
    \left| \<\psi^{(1)}(t),B\psi^{(1)}(t)\> - \<\psi^{(2)}(t),B\psi^{(2)}(t)\>\right| \le \frac1h\int_0^t |\rho(s,t)| \,ds,
    \]
    with
    $
    \rho(s,t) = \<\psi^{(1)}(s),\left[e^{iA_2(t-s)/h}Be^{-iA_2(t-s)/h},A_1-A_2\right]\psi^{(1)}(s)\>$, 
 where we have denoted by  $[\mathcal A,\mathcal B]=\mathcal{ AB-BA}$ the standard commutator.
\end{lemma}

\begin{proof}
We denote the unitary evolution operators by $U_j(t) = e^{-iA_jt/h}$ and calculate
\begin{align*}
&\<\psi^{(1)}(t),B\psi^{(1)}(t)\> - \<\psi^{(2)}(t),B\psi^{(2)}(t)\>\\
&=\<U_1(t)\psi_0,BU_1(t)\psi_0\> - \<U_2(t)\psi_0,BU_2(t)\psi_0\> \\
&=\int_0^t \frac{d}{ds} \<\psi_0,U_1(s)^*U_2(t-s)^*BU_2(t-s)U_1(s)\psi_0\> ds\\
&=\frac{1}{ih}\int_0^t \<\psi_0,U_1(s)^*[U_2(t-s)^*BU_2(t-s),A_1-A_2]U_1(s)\psi_0\> ds\\
&=\frac{1}{ih}\int_0^t \<\psi^{(1)}(s),[U_2(t-s)^*BU_2(t-s),A_1-A_2]\psi^{(1)}(s)\> ds.
\end{align*}
\end{proof}


\section{Proof of error estimates: partially flat coupling}
\label{sec:error-scales}

In this section, we prove error estimates in $L^2$-norm for general
potentials $W$  (Remark~\ref{rem:sigmaxy}), in $H^1$-norm (Proposition~\ref{prop:error-scales-energy}), and for quadratic observables (Proposition~\ref{prop:observables}).

\subsection{Proof of \textcolor{black}{Remark~\ref{rem:sigmaxy}}}
\label{sec:proof-error-scales}

  To prove the estimates of Example~\ref{ex:W}, recall that we have denoted
\[
\eta =     \|\<x\>^{-\sigma_x}\langle y\rangle
      ^{-\sigma_y}\nabla _yW\|_{L^\infty}<\infty.
  \]
  The Fundamental Theorem of Calculus yields
  \[
    W(x,y)-W(x,y') = (y-y')\cdot\int_0^1\nabla_y W\(y'+\theta(y-y')\)d\theta,
  \]
  so we have
  \begin{align*}\nonumber
    | W(x,y)-W(x,y') | &\le |y-y'| \<x\>^{\sigma_x}\eta \int_0^1\<
                         y'+\theta(y-y')\>^{\sigma_y}d\theta\\\nonumber
    &\le |y-y'| \<x\>^{\sigma_x}\max \(\<y\>,
    \<y'\>\)^{\sigma_y} \eta,
  \end{align*}
  and we replace the pointwise estimate of
  $\delta W$ with
  \[
    |\delta W(x,x',y,y')| \le |y-y'| \(\<x\>^{\sigma_x} + \<x'\>^{\sigma_x}\)
   \max \(\<y\>,
    \<y'\>\)^{\sigma_y} \eta.
  \]
  The estimate on $\Sigma_\phi$ becomes
  \begin{align*}\nonumber
    \|\Sigma_\phi\|_{L^2}^2  \le \eta^2 \int_{\R^n}\(\int_{\R^n}\(
    \<x\>^{\sigma_x} +
    \<x'\>^{\sigma_x}\)|\phi^x(t,x')|^2dx'\)^2|\phi^x(t,x)|^2dx/\|\varphi_0^x\|^4_{L^2_x}\times
    \\\nonumber
     \times
 \underbrace{   \int_{\R^{d}}\(\int_{\R^d}|y-y'|\max\(\<y\>^{\sigma_y} ,
    \<y'\>^{\sigma_y}\)
    |\phi^y(t,y')|^2dy'\)^2|\phi^y(t,y)|^2dy/\|\varphi_0^y\|^4_{L^2_y}
      }_{=: \epsilon_y(t)} ,
  \end{align*}
  and we conclude by resuming the same estimates as above:
  \[
   \int_{\R^n}\(\int_{\R^n}\(    \<x\>^{\sigma_x} + \<x'\>^{\sigma_x}\)
     |\phi^x(t,x')|^2dx'\)^2|\phi^x(t,x)|^2\frac{dx}{
     \|\varphi_0^x\|^4_{L^2_x}}
     \le 4\|\<x\>^{\sigma_x} \phi^x(t)\|_{L^2_x}^2,
  \]
  and, in view of the inequality
  \[|y-y'|\max\(\<y\>^{\sigma_y} ,
    \<y'\>^{\sigma_y}\)\le 2\max\(\<y\>^{\sigma_y} |y|,
    \<y'\>^{\sigma_y}|y'|\)\le 2\(\<y\>^{\sigma_y} |y|+
    \<y'\>^{\sigma_y}|y'|\),\]
  we find
  \[
    \epsilon_y(t)\le 16 \|\<y\>^{\sigma_y}|y| \phi^y(t)\|_{L^2_y}^2.
  \]
\hfill $\square$

\subsection{Proof of Proposition~\ref{prop:error-scales-energy}}
\label{sec:proof-error-scales-energy}


To prove error estimates in $H^1(\R^{n+d})$, we
differentiate \eqref{eq:remainder-flat} in space, and two aspects
must be considered: (i) In our framework, the operator $\nabla_{x,y}$ does not commute with $H$. 
(ii) We must estimate $\nabla_{x,y} \Sigma_\psi$ and $\nabla_{x,y} \Sigma_\phi$. 
Indeed, we compute
\[
  i\d_t \nabla_x r_\psi = H\nabla_x r_\psi +
  [\nabla_x,H]r_\psi+\nabla_x \Sigma_\psi,
\]
and
\[
  [\nabla_x,H]=\nabla_x H-H\nabla_x = \nabla_x V_1+\nabla_x W. 
\]
In the typical case where $V_1$ is harmonic, $\nabla_x V_1$ is linear
in $x$, and so $xr_\psi$ appears as a source term. Note that in the
general setting of Assumption~\ref{hyp:potentials}, $|\nabla_x
V_1(x)|\lesssim \<x\>$.
\begin{remark}
 If $\nabla_x V_1$ and
 $\nabla_x W$ are bounded, then
 Lemma~\ref{lem:energy-estimate} yields
 \[
  \|\nabla_x r_\psi(t)\|_{L^2} \le  \int_0^t \(C\|r_\psi(s)\|_{L^2}+
   \|\nabla_x\Sigma_\psi(s)\|_{L^2}\)ds.
 \]
The term $\|r_\psi(s)\|_{L^2}$ is estimated in
 Proposition~\ref{prop:error-scales}, and $
 \|\nabla_x\Sigma_\psi(s)\|_{L^2}$ is estimated below.
\end{remark}
  Multiplying \eqref{eq:remainder-flat}  by $x$,
we find similarly
\[
  i\d_t (x r_\psi) = H (x r_\psi )+
  [x,H]r_\psi+x \Sigma_\psi= H (x r_\psi )+\nabla_x r_\psi +x
  \Sigma_\psi. 
\]
Energy estimates provided by Lemma~\ref{lem:energy-estimate} applied to
the equation for $\nabla_x r_\psi$ and $x r_\psi$ then yield a closed
system of estimates:
\begin{align*}
  \|\nabla_x r_\psi(t)\|_{L^2} +  \|x r_\psi(t)\|_{L^2} &\le
  \int_0^t\(\|(\nabla_x V_1 +\nabla_x W)
  r_\psi(s)\|_{L^2} + \|\nabla_x r_\psi(s)\|_{L^2}\)ds\\\
  &\quad+ \int_0^t\(\|\nabla_x
    \Sigma_\psi(s)\|_{L^2}+ \|x \Sigma_\psi(s)\|_{L^2}\)ds\\
  &\le C  \int_0^t\(\|x
  r_\psi(s)\|_{L^2} + \|\nabla_x r_\psi(s)\|_{L^2}\)ds\\
  &\quad+ \int_0^t\(\|\nabla_x
    \Sigma_\psi(s)\|_{L^2}+ \|x \Sigma_\psi(s)\|_{L^2}\)ds,
\end{align*}
where we have used the estimate $|\nabla_x V_1 +\nabla_x W|\le
C(1+|x|)$, and the uncertainty principle (uncertainty in $x$,
Cauchy-Schwarz in $y$),
\[
  \|f\|_{L^2}^2\le \frac{2}{n}\|\nabla_x f\|_{L^2} \|x f\|_{L^2} .
\]
The Gronwall Lemma then yields
\[
  \|\nabla_x r_\psi(t)\|_{L^2} +  \|x r_\psi(t)\|_{L^2} \le \int_0^t e^{Cs}\(\|\nabla_x
    \Sigma_\psi(s)\|_{L^2}+ \|x \Sigma_\psi(s)\|_{L^2}\)ds,
  \]
  for some $C>0$. We compute
  \[
    \nabla_x\Sigma_\psi = \(\nabla_x W(x,y)-\nabla_x W(x,0)\)\psi_{\rm
      app} + \delta W(x,0,y,0)\nabla_x \psi_{\rm app}.
  \]
  The first term in controlled by $|y|\|\nabla_x\nabla_y
  W\|_{L^\infty} |\psi_{\rm app}|$. The second term is controlled like
  in Section~\ref{sec:error-flat}, by replacing $\psi_{\rm app}$ with
  $\nabla_x \psi_{\rm app}$. We can of course resume the same approach
  when considering $\nabla_y r_\psi$, and the analogue of the above
   first term is now controlled by $|x| \|\nabla_x\nabla_y
  W\|_{L^\infty} |\psi_{\rm app}|$. Finally, in the case of $r_\phi$,
  computations are similar (we do not keep track of the precise
  dependence of multiplicative constants here).

  \subsection{Proof of Proposition~\ref{prop:observables}}
\label{sec:proof-observables}
{\bf Proof.}
  We use Lemma~\ref{lem:quad-estimate} for the operators $H$ and the approximate Hamiltonian $H_{\rm bf}$,
\begin{equation}\label{eq:Hbf}
H_{\rm bf} = H_x + H_y + W(x,0) + W(0,y) - W(0,0),
\end{equation}
to obtain
\[
|e_\psi(t)| \le \int_0^t |\rho_\psi(t,s)| \,ds,
\]
where
\[
\rho_\psi(t,s) = 
\<\psi_{\rm app}(s),\left[B(t-s),H-H_{\rm bf}\right]\psi_{\rm app}(s)\>,\quad 
B(\sigma) = e^{i\sigma H}B e^{-i\sigma H}.
\]
By Egorov Theorem, see \cite[Theorem~11.1]{Zworski}, the operator $B(\sigma)$ is also a pseudodifferential operator, that is, $B(\sigma)={\rm op}(b(\sigma))$ for some function $b(\sigma)$ that satisfies the growth condition \ref{def:observables}.  We have
\[
H-H_{\rm bf} = W(x,y) - W(x,0) - W(0,y) + W(0,0)=\delta W(x,0,y,0)=:\delta W(x,y),
\]
with the notations of~\ref{sec:proof-error-scales}. Then, by the
direct estimate of Lemma~\ref{lem:comm} below,
\begin{equation*}
  \|[B(\sigma),\delta W]\psi_{\rm app}(s)\|_{L^2} \le 
  C_{b(\sigma)} 
  \(\|\nabla (\delta W)\psi_{\rm app}(s)\|_{H^1} + C_2(\delta W) 
  \|\psi_{\rm app}(s)\|_{L^2}\),
\end{equation*}
where $C_{b(\sigma)}>0$ depends on derivative bounds for the function $b(\sigma)$ and
\[
C_2(\delta W) = \sum_{2\le|\alpha|\le N_{n+d}}\|\partial^\alpha \delta W\|_{L^\infty}.
\]
We therefore obtain 
\[
|\rho_\psi(t,s)| \le \,C_{b(t-s)}\,  \Big( \|\nabla (\delta W)\psi_{\rm app}(s) \|_{H^1} 
+ C_2(\delta W)\, \|\psi_0\|_{L^2}\Big) \,\| \psi_0\|_{L^2}.
\]
Using the rectangular $n\times d$ matrix $M(x,y)$ introduced in Remark~\ref{rem:collocation}, 
the gradient of $\delta W(x,y)$ can be written as
\[
\nabla(\delta W)(x,y) = 
\left(\begin{array}{c}\nabla_x W(x,y) - \nabla_x W(x,0)\\*[1ex] \nabla_y W(x,y) - \nabla_yW(0,y)\end{array}\right)
= \left(\begin{array}{c}\int_0^1 M(x,\eta y)y\,d \eta\\*[1ex] \int_0^1 \,^t M(\theta x,y)x \,d\theta\end{array}\right)
\]
We estimate the Sobolev norm by
\begin{align*}
\|\nabla(\delta W)\psi_{\rm app}(s)\|_{H^1} &\le 
\|\nabla M\|_{L^\infty} \left(\|x\psi_{\rm app}(s)\|_{L^2}+\|y\psi_{\rm app}(s)\|_{L^2}\right) \\
&\quad+ \|M\|_{L^\infty}  \left(\|\nabla(x\psi_{\rm app}(s))\|_{L^2}+\|\nabla(y\psi_{\rm app}(s))\|_{L^2}\right), 
\end{align*}
so that integration in time provides
\begin{align*}
|e_\psi(t)| &\le C_b \|\nabla M\|_{L^\infty} \|\psi_0\|_{L^2} 
\int_0^t \left(\|x\psi_{\rm app}(s)\|_{L^2}+\|y\psi_{\rm app}(s)\|_{L^2}\right)ds \\
&\quad+ C_b \|M\|_{L^\infty} \|\psi_0\|_{L^2}  \int_0^t \left(\|\nabla(x\psi_{\rm app}(s))\|_{L^2}+\|\nabla(y\psi_{\rm app}(s))\|_{L^2}\right) ds\\
& \quad+ C_b\, C_2(\delta W)\, t \, \|\psi_0\|_{L^2}^2, 
\end{align*}
where the constant $C_b = \max_{\sigma\in[0,t]} C_{b(\sigma)}$ depends on derivatives of $b$.
In the mean-field case, the approximate Hamiltonian is time-dependent, 
\begin{equation}\label{def:Happc}
H_{\rm mf}(t) = H_x + H_y + \<W\>_y(t) +  \<W\>_x(t) - \<W\>(t).
\end{equation}
The difference of the Hamiltonians is also a function, which is now time-dependent, $H-H_{\rm mf}(t)= W+\langle W\rangle (t) -\langle W\rangle_x (t) -\langle W\rangle_y (t)$.
However, it is easy to check that a similar estimate can be performed, leading to an analogous conclusion. 
\hfill $\square$

\subsection{Commutator estimate}\label{sec:commutator}

We now explain the commutator estimate used in the previous subsection:
\begin{lemma}\label{lem:comm}
Let $N\ge 1$ and $b=b(z,\zeta)$ be a smooth function on $\R^{2N}$ satisfying the H\"ormander growth condition~\ref{def:observables}. Let $\delta W$ be a smooth function on $\R^N$ with bounded derivatives. 
Then, there exist constants $C_b>0$ and $M_N>0$ such that
\[
\|[\op(b),\delta W]\psi\|_{L^2} \le C_b\, \Big(\|\nabla(\delta W)\psi\|_{H^1} + \sum_{2\le|\alpha|\le M_N}\|\partial^\alpha(\delta W)\|_\infty \|\psi\|_{L^2}\Big)
\]
for all $\psi\in H^1(\R^N)$. 
\end{lemma}
\begin{proof}
We explicitly write the commutator as
\begin{align*}
&[\op(b),\delta W]\psi(z) = \\
&(2\pi)^{-N} \int_{\R^{2N}} 
b\left(\frac {z+z'}{2} ,\zeta\right) e^{i\zeta\cdot (z-z')} \left(\delta W(z')-\delta W(z)\right) \psi(z') \,d\zeta dz'.
\end{align*}
We Taylor expand the function $\delta W(z)$ around the point $z'$, so that
\[
\delta W(z)-\delta W(z') =\nabla(\delta W)(z')\cdot(z-z') + (z-z')\cdot \delta R_2(z,z')(z-z')
\] 
with
\[
\delta R_2(z,z') = \int_0^1 (1-\vartheta)\nabla^2(\delta W)(z'+\vartheta(z-z')) \,d\vartheta.
\]
Corresponding to the above decomposition, we write
$[\op(b),\delta W]\psi(z) = f_1(z) + f_2(z)$
and estimate the two summands separately. 
We observe that $(z-z') e^{i\zeta\cdot (z-z')} = -i\nabla_\zeta e^{i\zeta\cdot (z-z')}$ and perform an integration by parts to obtain
\begin{align*}
&\int_{\R^{2N}} 
b\left(\frac {z+z'}{2} ,\zeta\right) e^{i\zeta\cdot (z-z')}\,\nabla(\delta W)(z')\cdot(z-z')\, \psi(z') \,d\zeta dz' \\
&=i\int_{\R^{2N}} 
\nabla(\delta W)(z')\cdot \nabla_\zeta b\left(\frac {z+z'}{2} ,\zeta\right) e^{i\zeta\cdot (z-z')}  \psi(z') \,d\zeta dz'
\end{align*}
Therefore, 
\[
\|f_1\|_{L^2}\le  C_b\,\|\nabla(\delta W)\psi\|_{H^1},
\]
where the constant $C_b>0$ depends on derivative bounds of the function $b$. For the remainder 
term of the above Taylor approximation we write
\begin{align*}
&\int_{\R^{2N}} 
b\left(\frac {z+z'}{2} ,\zeta\right) e^{i\zeta\cdot (z-z')}\,(z-z')\cdot \delta R_2(z,z')(z-z')\, \psi(z') \,d\zeta dz' \\
&=\int_{\R^{2N}} 
{\rm tr}\left(\delta R_2(z,z')\nabla_\zeta^2 b\left(\frac {z+z'}{2} ,\zeta\right)\right) e^{i\zeta\cdot (z-z')}  \psi(z') \,d\zeta dz',
\end{align*}
and obtain that
\[
\|f_2\|_{L^2} \le 
C_b' \sum_{2\le|\alpha|\le M_N}\|\partial^\alpha(\delta W)\|_\infty \|\psi\|_{L^2},
\]
where $C_b'>0$ depends on derivative bounds of $b$, and $M_N>0$ depends on the dimension~$N$.
\end{proof}
\subsection{Proof of energy conservation}\label{sec:proof-energy}

Here we provide an elementary ad-hoc proof for energy conservation of the time-dependent Hartree 
approximation, Lemma~\ref{lem:energy}.

\begin{proof}
 A first observation is that 
\[
\<\phi_{\rm app}(t),H_{\rm mf}(t)\phi_{\rm app}(t)\> = \<\psi_0,H_{\rm mf}(0)\psi_0\>
\quad {\rm  for\  all}\ t\ge 0.
\]
Indeed, 
\begin{align*}
\frac{d}{dt}\<\phi_{\rm app}(t),H_{\rm mf}(t)\phi_{\rm app}(t)\> &= 
\<\phi_{\rm app}(t),\partial_t H_{\rm mf}(t)\phi_{\rm app}(t)\> \\
&= \<\phi_{\rm app}(t),\partial_t W_{\rm app}(t)\phi_{\rm app}(t)\>
\end{align*}
with 
$\displaystyle{W_{\rm app}(t)= \langle W\rangle_y(t)+ \langle W\rangle_x(t)-\<W\>(t).}$
We deduce 
\begin{align*}
&\frac{d}{dt}\<\phi_{\rm app}(t),H_{\rm mf}(t)\phi_{\rm app}(t)\> \\
&= 
\int W(x,y) \left(\d_t|\phi^x(t,x)|^2\ |\phi^y(t,y)|^2 + |\phi^x(t,x)|^2\ \d_t |\phi^y(t,y)|^2\right) dxdy\\ 
&\quad - \int W(x,y)\d_t\left(|\phi^x(t,x)|^2\ |\phi^y(t,y)|^2\right)dx dy = 0,
\end{align*}
where we have used the self-adjointness of $H_{\rm mf}(t)$ and norm-conservation in the multiplicative components. 
Secondly, since
\begin{align*}
\<\psi_0,W_{\rm app}(0)\psi_0\> &= \<\psi_0,\left(\<W\>_x(0) + \<W\>_y(0) - \<W\>(0)\right)\psi_0\>\\
&= 2\<W\>(0) - \<W\>(0) = \<W\>(0),
\end{align*}
the approximate energy coincides with the actual energy, and we 
obtain the aimed for result.
\end{proof}


\section{Proof of error estimates in the semiclassical r\'egime}
\label{sec:proof-error-semi}
Here we present the proofs of the semiclassical estimates given in Proposition~\ref{prop:error-semi} and Proposition~\ref{prop:observables-semi}.

\subsection{Error estimates for the wave function} \label{subsec:error-semi}

In  this section, we prove Proposition~\ref{prop:error-semi}, and 
comment on the constants $K_0,K_1$, which may be analyzed more
explicitly in some cases. 

\subsubsection{Approximation  by partial Taylor expansion} $ $

\bigskip\noindent
{\bf Proof.}
Section~\ref{sec:taylor-semi} defines an approximate solution of the from  
\[
\psi_{\rm app}^\eps(t,x,y) = \eps^{-d/4}  
e^{(ip(t)\cdot (y-q(t))/\eps +iS(t)/\eps}  u_2\(t,\frac{y-q(t)}{\sqrt \eps}\)
 \psi_1^\eps(t,x)
\]
with
\[
  i\eps\d_t \psi_1^\eps + \frac{1}{2}\Delta_x \psi_1 ^\eps=
  \(V_1(x)+W(x,q)\)\psi_1^\eps\quad;\quad \psi^\eps_{1\mid t=0} = \varphi_0^x
  \]
and
\begin{align*}
 & i\d_t u_2 + \frac{1}{2}\Delta_z u_2 = \frac{1}{2}\<z,\nabla^2
  V_2\(q\)z\>u_2\quad ;\quad u_{2\mid t=0} = a,\\
 & \dot q = p,\, q(0)=q_0,\quad \dot p = -\nabla V_2(q),\;    p(0)=p_0,
\;\; S(t) = \int_0^t \(\frac{|p(s)|^2}{2}-V_2(q(s))\)ds.
  \end{align*}
  It  solves the equation
\[
   i\eps \d_t \psi_{\rm app}^\eps   +\frac{1}{2}\Delta_x
         \psi_{\rm app}^\eps +\frac{\eps^2}{2}\Delta_y\psi_{\rm app}^\eps
          -V \psi_{\rm app}^\eps =\eps^{-d/4}
          e^{(ip\cdot z/\sqrt\eps +iS/\eps}\( r_1^\eps +r_2^\eps\),
\]
where the remainder $r_1^\eps$ is due to the Taylor expansion in $V_2$, and
satisfies the pointwise estimate
\[
|  r_1^\eps(t,x,z) |\le  \frac{1}{6} \times \eps^{3/2} \|\nabla^3 V_2\|_{L^\infty_y}
|\psi_1^\eps(t,x)|\times |z|^3 |u_2(t,z)|
\]
while the remainder $r_2^\eps$ is due to the Taylor expansion in $W$, and
satisfies the pointwise estimate
\[
|  r_2^\eps(t,x,z) |\le  \sqrt\eps \  \|\nabla_y W\|_{L^\infty}
|\psi_1^\eps(t,x)|\times| z| |  u_2(t,z)|.
\]
This implies for the $L^2$-norm,
\begin{align*}
  & \|r_1^\eps(t)\|_{L^2}\le \frac{\eps^{3/2}}{6} \|\nabla^3
    V_2\|_{L^\infty_y} \|\varphi_0^x\|_{L^2_x}\||z|^3
    u_2(t)\|_{L^2_z},\\
   & \|r_2^\eps(t)\|_{L^2}\le  \sqrt\eps\, \|\nabla_y W\|_{L^\infty} \|\varphi_0^x\|_{L^2_x}\|z
    u_2(t)\|_{L^2_z}.
\end{align*}
Lemma~\ref{lem:energy-estimate} then yields, with now $h=\eps$,
\begin{align*}
  \|\psi^\eps(t)- \psi_{\rm app}^\eps(t)\|_{L^2}& \le \frac{\sqrt\eps}{6} 
 \|\nabla^3    V_2\|_{L^\infty_y}  \|\varphi_0^x\|_{L^2_x} \int_0^t
  \||z|^3 u_2(s)\|_{L^2_z}ds \\
&\quad  +\frac{\|\nabla_y W\|_{L^\infty}}{\sqrt\eps}  
  \|\varphi_0^x\|_{L^2_x}  \int_0^t\| |z|
  u_2(s)\|_{L^2_z}ds .
\end{align*}
According to the signature of $\nabla^2 V_2(q(t))$, the quantities
$\||z|^3 u_2(s)\|_{L^2_z}$ and $\| |z|
  u_2(s)\|_{L^2_z}$ may be bounded uniformly in $s\ge 0$ or
  not. For instance, they are bounded if $\nabla^2 V_2$ is uniformly
  positive definite, or at least uniformly positive definite along the
  trajectory $q$. On the other hand, we always have an exponential
  bound, even if it may not be sharp,
  \[
    \||z|^3 u_2(s)\|_{L^2_z}+\| |z|
  u_2(s)\|_{L^2_z}\le C_0 e^{C_1s},
\]
for some constants $C_0,C_1>0$. This control is sharp in the case
where $\nabla^2 V_2$ is uniformly
  negative definite. See
    e.g. \cite[Lemma~10.4]{CaBook2} for a proof of the exponential
    control, and \cite[Section~10.5]{CaBook2} for a discussion on its
    optimality. In particular, for bounded time intervals, the
  (relative) error is
  small if $\|\nabla_y W\|_{L^\infty}\ll \sqrt\eps\ll 1$.
  \hfill $\square$

  \begin{remark}
    If $\nabla_y W$ is not bounded, e.g. $\nabla_y W (x,y) = \eta
    \<x\>^\gamma$, then we can replace the previous error estimate with
   \begin{align*}
  \|\psi^\eps(t)- \psi^\eps_{\rm app}(t)\|_{L^2}& \le \frac{\sqrt\eps}{6} 
 \|\nabla^3    V_2\|_{L^\infty_y} 
 \|\varphi_0^x\|_{L^2_x}  \int_0^t
  \||z|^3 u_2(s)\|_{L^2_z}ds \\
&\quad  +\frac{\eta}{\sqrt\eps}  
  \int_0^t\|\<x\>^\gamma\psi_1^\eps(s)\|_{L^2_x}  \||z|
  u_2(s)\|_{L^2_z}ds .
\end{align*}
  In other words, the cause for the unboundedness of $  \nabla_y W$ is
  transferred to a weight for $\psi_1^\eps$. Similarly, if  $  \nabla_y W$ is
  unbounded in $y$, we may change the weight in the terms
  $\||z|^k u_2\|_{L^2_z}$, after substituting $y$ with $q+z\sqrt\eps$. 
  \end{remark}
  
\subsubsection{Approximation by partial averaging} $ $

\begin{proof}
The semiclassical approximation obtained by partial averaging reads:
\[
\psi_{\rm app}^\eps (t,x,y) = \eps^{-d/4} e^{ip(t)\cdot (y-q(t))/\eps +iS(t)/\eps}  u_2\(t,\frac{y-q(t)}{\sqrt \eps}\)
 \psi_1^\eps(t,x)
 \]
 with
 \[ 
   i\eps\d_t \psi_1 ^\eps+ \frac{1}{2}\Delta_x \psi_1^\eps =
  \(V_1(x)+\<W(x,\cdot)\>_y(t)\)\psi_1^\eps\quad;\quad \psi_{1\mid t=0} ^\eps= \varphi_0^x
  \]
  and
  \begin{align*}
 & i\d_t u_2 + \frac{1}{2}\Delta_z u_2 = \frac{1}{2}\ z\cdot\<\nabla^2
  V_2\>_y\!(t)z \ u_2\quad ;\quad u_{2\mid t=0} = a,\\
  &  \dot q = p,\quad q(0)=q_0,\quad \dot p = -\langle\nabla V_2\rangle_y(t),\quad
    p(0)=p_0,\;\;  \dot S(t) = \frac{|p(t)|^2}{2}- \langle V_2\rangle_y(t).
  \end{align*}
To estimate the size of $\tilde r_1$ and $\tilde r_2$ introduced in Section~\ref{sec:average-semi}, we might argue again via Taylor expansion. 
Indeed,
\begin{align*}\nonumber
\|a\|_{L^2}^2\<V_2\>_y &= \int V_2(q(t)+\sqrt\eps z)\  |u_2(t,z)|^2 dz \\\nonumber
&= 
V_2(q(t)) + \sqrt\eps\, \nabla V_2(q(t))\cdot\int z |u_2(t,z)|^2 dz + r_3^\eps(t), 
\end{align*}
where 
\[
|r_3^\eps(t)| \le \frac{\eps}{2} \|\nabla^2V_2\|_{L^\infty} \left\||z|
  u_2(t,z)\right\|^2_{L^2_z}. 
\]
Hence, we have for all averages $f = V_2,\nabla V_2,\nabla^2 V_2,W(x,\cdot)$ that
\[
\<f\>_ y (t)= f\(q(t)\) + \O(\sqrt\eps),
\]
where the error constant depends on moments of $|u_2|^2$. In particular, if $u_2(0)$ is Gaussian, 
the odd moments of $|u_2(t,z)|^2$ vanish, and the above estimate improves to $\O(\eps)$. 
Hence, the $L^2$-norm of $\tilde r_1$ is $\O(\sqrt\eps)$  close to
the $L^2$-norm of $r_1$, and the $L^2$-norm of $\tilde r_2$ is
$\O(\eta\sqrt\eps)$, $\eta= \|\nabla_y W\|_{L^\infty}$,  close to 
the $L^2$-norm of $r_2$ (with each time an extra $\sqrt\eps$ gain in the above
mentioned Gaussian case). In particular, the order of magnitude for the
difference between exact and approximate solution is the same as in
the previous subsection, only multiplicative constants are
affected. We emphasize that the constants $C_0$ and $C_1$ from the
previous subsection are in general delicate to assess. On the other
hand, in specific cases (typically when $u_2$ is Gaussian and
$\nabla^2 V_2$ is known), they can be computed rather explicitly. 
\end{proof}

\subsection{Error estimates for quadratic  observables}\label{sec:error-quad-semi}

The proof of Proposition~\ref{prop:observables-semi} is discussed in the next two sections. 

\subsubsection{ Approximation by  Taylor expansion} $ $ 

\begin{proof}
  Taylor expansion yields a time-dependent Hamiltonian $ H_{\rm app} ^\eps= H^\eps_{\rm te}$ with 
\begin{align*}
 H^\eps_{\rm te}&:= -\frac{1}{2}\Delta_x-\frac{\eps^2}{2}\Delta_y
                +V_1(x) + W(x,q)+V_2(q) +(y-q)\cdot \nabla V_2(q)\\
  &\quad+
  \frac{1}{2}\<y-q,\nabla^2 V_2(q)(y-q)\>,
\end{align*}
where $q=q(t)$. In particular, the difference $  H^\eps- H^\eps_{\rm te}$ is a function,
\begin{align*}
  H^\eps- H^\eps_{\rm te} & = W(x,y)-W(x,q) +V_2(y) - V_2(q)-(y-q)\cdot \nabla
  V_2(q)\\
 &\quad - \frac{1}{2}\<y-q,\nabla^2 V_2(q)(y-q)\>=:\delta W(t,x,y).
\end{align*}
In view of  Lemma~\ref{lem:quad-estimate}, if $B={\rm op}_\eps(b)$
with $b\in C^\infty_c (\R^{2d})$, it
yields (a posteriori estimate)
\[
  \left| \<\psi^\eps(t),B\psi^\eps(t)\> -\<\psi^\eps_{\rm app}(t),B\psi^\eps_{\rm
  app}(t)\>\right| \le \frac{1}{\eps}\int_0^t |\rho^\eps(t,x)|ds,
\]
where
\[
  \rho^\eps(t,s) = \<\psi^\eps_{\rm app}(s), [B(t-s),\delta W (s)]\psi^\eps_{\rm
             app}(s)\>,\quad 
  B(\sigma) = e^{i\sigma H/\eps} B e^{-i\sigma H/\eps}.
\]
By Egorov Theorem~\cite[Theorem~11.1]{Zworski},  $B(\sigma) = \eps\, \op_\eps(b(\sigma))$ for a
function $b(\sigma)\in C^\infty_c(\R^d)$. Therefore,  
by semiclassical calculus,
\[
\frac{1}{i\eps}[B(t-s),\delta W(s)] = 
\op_\eps \(\{b(t-s),\delta W(s)\}\) + \eps^2 \op_\eps \(r^\eps(s,t)\),
\]
where $\|\op_\eps(r^\eps(s,t))\|_{\mathcal L(L^2)}$ is bounded uniformly in $\eps$, whence the estimate of Proposition~\ref{prop:observables-semi}.
\end{proof}
\subsubsection{Approximation by partial averaging} $ $

\begin{proof}
The time-dependent Hamiltonian is ${H}_{\rm app}^\eps={H}_{\rm pa}^\eps$ with 
\begin{align*}
{H}_{\rm pa}^\eps &= -\frac{1}{2}\Delta_x-\frac{\eps^2}{2}\Delta_y
                +V_1(x) +\langle W(x,\cdot)\rangle_y +\langle V_2\rangle_y(t) +(y-q)\cdot\langle  \nabla V_2\rangle_y(t)\\
  &\quad+
  \frac{1}{2}(y-q)\cdot \langle \nabla^2 V_2\rangle_y(t) (y-q),
\end{align*}
where $q=q(t)$. In particular,  as in the preceding case, the difference $  H^\eps- {H}_{\rm pa}^\eps$ is a time-dependent function
\begin{align*}
 H^\eps- {H}^\eps_{\rm pa}  &= W(x,y)- \langle W(x,\cdot)\rangle_y +V_2(y) -\langle V_2\rangle_y(t) \\
  &\quad -(y-q)\cdot\langle  \nabla V_2\rangle_y(t)
  - \frac{1}{2}(y-q)\cdot \langle \nabla^2 V_2\rangle_y(t) (y-q)=:\widetilde{ \delta W}(t,x,y),
\end{align*}
and the arguments developed above also apply. 
\end{proof}

\subsection{Time-adiabatic approximation}
The evolution equations for the quantum part of the system, equations~\eqref{eq:psi1} and~\eqref{eq:psi1av}, can be written as an adiabatic problem:
\[
i\eps \partial_t \psi_1^\eps(t) =  \mathfrak h(t) \psi_1^\eps(t),\;\; \psi_1^\eps(0)=\varphi_0^x,
\]
where $\mathfrak h(t)$ is one of the time-dependent self-adjoint operators on $L^2(\R^n)$
\[
\mathfrak h_{\rm te}(t) = -\frac 1 2 \Delta  +V_1(x) + W(x,q(t))\;\;\mbox{and}\;\;
\mathfrak h_{\rm pa}(t) = -\frac 1 2 \Delta  +V_1(x) + \langle W(x,.)\rangle_y(t).\]
We assume here that $\mathfrak h(t)$ has a compact resolvent and thus, that its spectrum consists in a sequence of time-dependent eigenvalues 
$$\eigen_1(t) \le \eigen_2(t)\le \cdots\le \eigen_k(t)\Tend{k}{ +\infty} \infty.$$
We also assume that some eigenvalue $\eigen_{j}(t)$ is separated from the remainder of the spectrum for all $t\in\R$ and that the initial datum $\varphi_0^x$ is in the eigenspace of $\mathfrak h(0)$ for the eigenvalue $\eigen_j(0)$:
\begin{equation} \label{eq:vp}
\mathfrak h(0) \varphi_0^x= \eigen_j(0) \varphi_0^x.
\end{equation}
Then adiabatic theory as developed by Kato~\cite{K1}  states that $\psi_1(t)$
stays in the eigenspace of $\eigen_j(t)$ on finite time, up to a phase. 

\begin{proposition}[Kato~\cite{K1}]\label{prop:Kato}
Assume we have~\eqref{eq:vp} and that $ \eigen_j(0)$ is a simple  eigenvalue of $\mathfrak h(0)$   such that there exists $\delta_0>0$ for which
\[d\left( \{ \eigen_j(t)\} , {\rm Sp} ( \mathfrak h(t)) \setminus \{\eigen_j(t)\}\right)\geq \delta_0.\]
Denote by $\Phi_j^x(t)$ a family of normalized eigenvectors of $\mathfrak h(t)$ such that 
\[\Phi_j^x(0)= \varphi_0^x,\quad\< \Phi^x_j(t), \partial_t \Phi_j^x(t)\>=0.\]
Then, for all  $T>0$, there exists a constant $C_T>0$ such that  
\[\left\| \psi_1^\eps(t) - e^{-\frac i\eps \int_0^t \eigen_j(s) ds } \Phi^x_j(t,x ) \right\| _{L^2_x} \le C_T \eps.\]
\end{proposition}

In contrast to the Born-Oppenheimer point of view recalled in Remark~\ref{rem:BO}, 
we obtain the following time-adiabatic extension for our wave-packet approximation:

\begin{corollary}\label{cor:BO}
In the setting of Proposition~\ref{prop:error-semi} and Proposition~\ref{prop:Kato}, we obtain the following approximate solution
\begin{align*}
\psi^\eps_{\rm app} (t,x,y)= 
e^{-\frac i\eps \int_0^t \eigen_j(s) ds +\frac i\eps S(t) +\frac i\eps p(t)\cdot (y-q(t))} \Phi^x_j(t,x)u_2\left(t,\frac{y-q(t)}{\sqrt\eps}\right).
\end{align*}
\end{corollary}

\section{Material for the numerical simulations}
All numerical calculations for the Section~\ref{sec:num} were run with the QUANTICS (Version 2.0) suite of programs for molecular quantum dynamics simulations [G. A. Worth, K. Giri, G. W. Richings, I. Burghardt, M. H. Beck, A. J\"ackle, and H.-D. Meyer. The QUANTICS Package, Version 2.0, (2020), University College London, London, U.K.]. 

\subsection{TDH and MCTDH calculations}
For each set of parameters we compared mean-field (TDH) calculations (one single-particle function per mode) to converged Multiconfiguration Time-Dependent Hartree (MCTDH) calculations (three single-particle functions per mode, i.e., nine configurations). The latter were compared to numerically ``exact" standard propagations on a two-dimensional grid, showing virtually no difference. In all cases, we chose for the primitive basis set a direct product of Gauss--Hermite functions ($128\times 1024$) associated to the local harmonic approximation of the potential energy centered at $(x, y) = (0, 0)$. We used default integrator schemes for both TDH and MCTDH: norm and energy preservation were dutifully checked. The initial datum of each situation was chosen as a quasi-coherent state corresponding to the local harmonic ground state (centered at $(x, y) = (0, 0)$, with no extra momentum, and with natural widths $1/\sqrt{2}$ in natural units). 

\subsection{Model set-up}
Calculations were run with dimensioned parameters and variables, all given numerically in terms of atomic units with realistic values chosen to be representative of what can be encountered in a typical organic molecule. The only quantity left with a dimension is time. The time scale used in both Figures is the natural time for the 
system coordinate~$x$, namely $t_1 =1/\omega_1 = \hbar/E_1$, which within our settings is equal to $109.74$ a.u. ($2.6544$ fs). 
We used a unit local curvature at the origin in natural units and a double-well width $\ell = 4$. In atomic units, 
where we chose $a_1 = 0.24536$ bohr, we have a dimensioned $a_1\cdot\ell = 0.98142$ bohr. 

\subsection{The reference model, blue curves}

\begin{description}
\item[Masses:] $\mu_1 = 1$ amu = 1822.9 m$_e$, $\mu_2 = 16$ amu = 29166 m$_e$, $\varepsilon=\sqrt{\mu_1/\mu_2}=1/4$
\item[Energies/frequencies:] $\omega_1 = 2000$ cm$^{-1} = 0.0091127$ hartree, 
$\omega_2 = 20$ cm$^{-1} = 0.000091127$ hartree , $\varpi = \omega_2/\omega_1=1/100$
\item[Natural scales for $x$:] $a_1 = 0.24536$ bohr (length),  $E_1 = 0.0091127$ hartree (energy), 
$t_1 = 109.74 \hbar$ hartree$^{-1} =2.6544$ fs (time)
\item[Double-well width:] $\ell = 4$, i.e., $L = 4 a_1 = 0.98142$ bohr 
\item[Width for the initial datum:] $\sigma_1 = a_1/\sqrt{2} = 0.17349$ bohr, $\sigma_2 = a_2/\sqrt{2} = 0.43373$ bohr,  where $a_2 = \epsilon/\sqrt{\varpi} a_1 = 10/4 a_1$ 
\item[Coupling constant] $\eta = -k_2/(3L) = -0.000082261$ a.u., where $k_2 = \mu_2\,\omega_2^2$ 
\end{description}

\subsection{Variations of the blue model}
\begin{description}
\item[Red curves:] softer bath, but same coupling relation $\eta = -k_2/(3L)$\\ 
$\omega_2 = 0.000022782$ a.u. such that $\varpi= 1/400$; 
$\sigma_2$ = 0.86747 bohr, $\eta = -0.0000051413$ a.u. both depend on $\omega_2$
\item[Gray curves:] same bath but stronger coupling $\eta = -3/8\, k_2/L$\\ 
$\varpi, \omega_2, \sigma_2$ unchanged against original model, 
$\eta = -0.000092543$ a.u. 
\item[Yellow curves:] same bath but weaker coupling $\eta = -k_2/(4L)$\\ 
$\varpi, \omega_2, \sigma_2$ unchanged against original model, 
$\eta = -0.000030848$ a.u. 
\end{description}



\bibliographystyle{abbrv}
\bibliography{biblio}


\end{document}